\documentclass[11pt, a4paper, fleqn]{article}
\usepackage[T1]{fontenc}
\usepackage[utf8]{inputenc}
\usepackage{newunicodechar}
\newunicodechar{ő}{\H{o}}
\newunicodechar{Ő}{\H{O}}
\newunicodechar{é}{\'e}
\newunicodechar{É}{\'E}
\usepackage{geometry}
\usepackage{amsmath,amssymb,amsthm}
\usepackage{enumitem}
\usepackage{booktabs}
\usepackage{graphicx}
\usepackage[unicode=true,
 bookmarks=false,
 breaklinks=false,pdfborder={0 0 1},backref=section,colorlinks=false]
 {hyperref}

\numberwithin{equation}{section}
\numberwithin{figure}{section}

\theoremstyle{plain}
\newtheorem{theorem}{Theorem}[section]
\newtheorem{lemma}[theorem]{Lemma}
\newtheorem{proposition}[theorem]{Proposition}
\newtheorem{corollary}[theorem]{Corollary}
\newtheorem{observation}[theorem]{Observation}

\theoremstyle{definition}
\newtheorem{definition}[theorem]{Definition}
\newtheorem{remark}[theorem]{Remark}

\newcommand{\dotcup}{\mathrel{\dot{\cup}}}
\newcommand{\dotsub}{\mathrel{\dot{\setminus}}}

\newcommand{\Par}{\operatorname{par}}
\newcommand{\cld}{\operatorname{cld}}
\newcommand{\lev}{\operatorname{level}}
\newcommand{\PTG}{\operatorname{PTG}}
\newcommand{\CC}{\mathbb{C}}
\newcommand{\FF}{\mathbb{F}}
\newcommand{\NN}{\mathbb{N}}
\newcommand{\QQ}{\mathbb{Q}}
\newcommand{\RR}{\mathbb{R}}
\newcommand{\ZZ}{\mathbb{Z}}

\title{The Non-Cancelling-Intersections Conjecture Fails for Left-Linear Trees}
\author{Hermann Wilhelm \\ Technische Universität Ilmenau}
\date{\today}

\begin{document}

\maketitle

\begin{abstract}
First formulated in \cite{amarilli2024non}, the Non-Cancelling Intersections (NCI) conjecture is an open problem in combinatorics stating that any set union can be constructively built from its algebraically non-cancelling intersections using only disjoint unions and subset complements. Also in \cite{amarilli2024non}, two orthogonal possible strengthenings are proposed: using only left-linear trees, and using non-trivial intersections only positively or only negatively depending on the sign of their Möbius value. Here we show that using only left-linear trees, the conjecture is false (independent of the other strengthening). Our argument is non-constructive. We prove the existence of a counterexample, though it is of immense size.
\end{abstract}

\section{Introduction}\label{sec:introduction}

A basic question in combinatorics is how to express a union of sets in terms of its intersections. The classical answer is the inclusion--exclusion principle, which alternates additions and subtractions. Its defect is precisely that it subtracts: intermediate terms overshoot the target and are corrected only later, so the identity yields no direct combinatorial construction of the union, and the cancellations it relies on can be enormous even when the final answer is small.

The Non-Cancelling Intersections (NCI) conjecture, formulated in~\cite{amarilli2024non}, asserts that this cancellation can always be avoided. Informally, it states that the union of any finite family of sets can be built up from its \emph{algebraically non-cancelling} intersections --- those whose Möbius value is nonzero --- using only two operations: disjoint union, and the complement of a subset inside a superset. A witness to the conjecture for a given family is thus an expression in a restricted algebra, the \emph{dot-algebra} (Definition 3.2 in \cite{amarilli2024non}), which is naturally presented as a tree whose leaves are non-cancelling intersections.

Two orthogonal strengthenings of the conjecture are proposed in~\cite{amarilli2024non}. The first restricts the shape of the witness, requiring the tree to be \emph{left-linear}: every right child is a leaf. The second restricts the way intersections may be used, requiring each non-trivial intersection to occur only positively or only negatively, according to the sign of its Möbius value. Either strengthening, if true, would give a sharper and more usable form of the conjecture; both are open.

Here we show that the NCI conjecture is false when only left-linear trees are allowed in the dot-algebra representation, independently of whether the second strengthening is imposed. Our route to this is a reformulation. Searching for a left-linear dot-algebra expression for a given lattice can be recast as a one-player game, played by toggling elements of the lattice; a game of this kind already appears in~\cite{amarilli2019lighting}, and we call it the \emph{toggle game}. Whether the toggle game can be won on every lattice has been open. We show that it cannot: there exists a lattice on which the toggle game is unwinnable. Our argument is non-constructive. It shows that an unwinnable lattice exists
without exhibiting one.

We stress that the toggle game is a self-contained combinatorial object. Apart from the present ntroduction, the NCI conjecture, left-linear trees and the dot-algebra are used only in Section~\ref{sec:ConnectionDot}, where the translation between the two settings is made precise; outside that section they are mentioned only in the statement of Corollary~\ref{cor:main} and in the open problems. A reader interested only in the toggle game may therefore skip that section entirely, and read the rest of the paper without reference to any of these notions.

We introduce a game called the \emph{plane toggle game}. The plane toggle game is played on the affine plane $\FF_p^2$ (with $p$ prime), whose points can be toggled. In contrast to the toggle game on lattices, choosing a point in the plane toggle game toggles only that point and no other points. However, in the plane toggle game there are constraints for every line. Every line has a special set of \emph{marked} points. The marked points are given from the start (they will be induced by the lattice); they do not change during the game and are independent of the toggled points. Notably, a point can be marked for one line and unmarked for another line. The constraint is the following: whenever a line contains at least two toggled points, at least one toggled point on that line must be marked for that line. The exact definition is given in Section~\ref{sec:PointToggle}. The relevant property is that every plane toggle game in which exactly $\lceil \sqrt{2p} \rceil + 1$ points are marked for every line is induced by a certain lattice, and winnability of the toggle game on that lattice implies winnability of the corresponding plane toggle game.

So our goal is to show that there is a prime number $p$ and a plane toggle game with $\lceil \sqrt{2p} \rceil + 1$ marked points on every line that cannot be won. To show this we use a probabilistic argument. For a random marking (such that $\lceil \sqrt{2p} \rceil + 1$ points are marked for each line), the chance of a specific point being marked on a given line is of order $1 / \sqrt p$. Using a finite version of the Erdős--Beck theorem, we show that $n$ points in the affine plane, not almost all of which are collinear, determine $\Omega(n^2)$ \emph{short} traces (traces of size at most a fixed constant). A trace of a set of points $T$ on a line $\ell$ is the set $T \cap \ell$, provided $|T \cap \ell| \geq 2$. So a short trace comes from a line that does not contain many points of $T$, but at least two. The chance of finding a marked point in each of $\Omega(n^2)$ short traces is of order $p^{-\Omega(n^2)}$. Since there are only $O(p^{2n})$ point sets of size $n$, a probabilistic argument shows that the expected number of admissible sets $T$ of size $n$ is less than one for suitable $n$ and large $p$. Since every winning sequence in the plane toggle game toggles one point at a time, at some point exactly $n$ points must be toggled on, and we arrive at a contradiction.

It remains to consider the case where one line contains almost all of the toggled points when exactly $n$ points are toggled on. Note that it is always possible to toggle on all points of a given line by first toggling a point that is marked for that line and then all other points on that line in any order. However, at some moment a prescribed number of points off that main line must be toggled on, and choosing the parameters accordingly we also get a contradiction in that case.

\section{Preliminaries}\label{sec:preliminaries}

\paragraph{Posets}
Let $P$ be a partially ordered set (poset). All posets in this paper are finite. If $x<y$, we say $y$ is an \emph{ancestor} of $x$ and $x$ is a \emph{descendant} of $y$. If $x<y$ and there is no element $z$ with $x<z<y$, then we say that $y$ is a \emph{parent} of $x$ and $x$ is a \emph{child} of $y$. If the poset is clear from the context and $x$ is one of its elements, we write $\Par(x)$ for the set of parents of $x$, $\cld(x)$ for the set of children of $x$, and
\[
  \uparrow\! x \;=\; \{\, y \in P : y \ge x \,\},
  \qquad
  \downarrow\! x \;=\; \{\, y \in P : y \le x \,\}
\]
for the set consisting of $x$ together with all its ancestors, respectively of $x$ together with all its descendants. Note that $x$ itself belongs to both $\uparrow\! x$ and $\downarrow\! x$. We represent $P$ by a DAG (directed acyclic graph) in the usual way: an edge from $x$ to $y$ indicates that $y$ is a parent of $x$. We omit edge arrows, since arrows always point upwards. The DAG representation defines a poset by taking the transitive closure of the edges. Figure~\ref{fig:toggle} shows an example of a poset and its DAG representation.

\begin{figure}[htbp]
  \centering
  \includegraphics[width=0.4\textwidth]{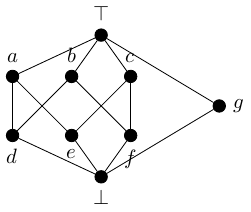}
  \caption{A representation of a poset.}
  \label{fig:toggle}
\end{figure}

The \emph{level} $\lev(v)$ of a vertex $v$ is the length $k$ of a longest path $(v_{0},v_{1},\ldots,v_{k})$ in the DAG representation starting from $v$, so $v_{0}=v$, where $v_{i+1}$ is a parent of $v_{i}$. In particular $\lev(\top)=0$, and if $x<y$ then $\lev(x)>\lev(y)$: a longest path from $y$ to $\top$ can be extended by a saturated chain from $x$ to $y$ of length at least $1$. Consequently every edge of the DAG goes from a vertex of larger level to a vertex of strictly smaller level, and \emph{two distinct vertices of the same level are incomparable}.

\paragraph{Lattices and the Möbius function}
A lattice is a poset in which every pair of elements $x,y$ has a unique least upper bound $x\vee y$ (the \emph{join}) and a unique greatest lower bound $x\wedge y$ (the \emph{meet}). Since all posets and thus all lattices here are finite, this implies the existence of a greatest element $\top$ and a smallest element $\bot$. The Möbius function is defined on the vertices of a lattice as follows: $\mu(\top)=1$, and for all other vertices $v$,
\[
  \mu(v)=-\sum_{u>v}\mu(u).
\]

Note that the poset depicted in Figure~\ref{fig:toggle} is in fact a lattice. Its Möbius values are as follows.
\begin{itemize}
    \item $\mu(\top)=1$.
    \item $\mu(v)=-1$ for $v\in\{a,b,c,g\}$ (vertices on level $1$ always have Möbius value $-1$).
    \item $\mu(d)= - ( \mu(\top) + \mu(a) + \mu(b) ) = 1$, and analogously $\mu(e)=\mu(f)= 1$.
    \item $\mu(\bot) = - (\mu(\top) + \mu(a) + \mu(b) + \mu(c) + \mu(d) + \mu(e) + \mu(f) + \mu(g) ) = 0$.
\end{itemize}

\paragraph{Prime fields}
For a prime number $p\in \NN$, $\FF_p$ denotes the field with $p$ elements and $\FF_p^2$ denotes the vector space of dimension $2$ over $\FF_p$.

\paragraph{Lines and traces}
Let $K$ be a field. Any two distinct points $Q_1 ,Q_2 \in K^2$ lie on a unique line
\[
  \ell \;=\; \{\, Q_1 + i\,(Q_2-Q_1) \;:\; i \in K \,\}.
\]
This representation of $\ell$ is not unique: any two distinct points of $\ell$ describe the same line. To have a unique representation of all $p^2+p$ distinct lines of $\FF_p^2$, we denote the lines by
\begin{align*}
    L_{a,b} &=  \{ (i , ai+b) : i\in \FF_p \} & &\text{for } a,b \in \FF_p, \\
    L_{\infty,b} &=\{ (b , i) : i\in \FF_p \} & &\text{for } b   \in \FF_p .
\end{align*}
Given a set $T=\{Q_1,\dots,Q_n\}$ of $n$ \emph{distinct} points in $K^2$ and a line $\ell\subseteq K^2$, we call the set $\ell\cap T$ the \emph{trace} of $\ell$ on $T$, provided it has at least two elements. We keep the lines and their traces notationally apart and write
\begin{align*}
 \mathcal{L}(T) &\;=\; \bigl\{\, \ell \ :\ \ell \text{ a line in } K^2 \text{ with } |\ell\cap T|\ge 2 \,\bigr\},\\
 \mathcal{T}(T) &\;=\; \bigl\{\, \ell\cap T \ :\ \ell \in \mathcal L(T) \,\bigr\}
\end{align*}
for the set of \emph{lines determined by $T$} and the set of \emph{traces of $T$}, respectively. So $\mathcal L(T)$ is a set of lines, while $\mathcal T(T)$ is a set of subsets of $T$. The map $\ell\mapsto \ell\cap T$ is a bijection from $\mathcal L(T)$ onto $\mathcal T(T)$: it is surjective by the definition of $\mathcal T(T)$, and it is injective because two lines with the same trace share at least two points and therefore coincide. In particular
\begin{align}\label{eq:linesEqualTraces}
  |\mathcal L(T)| \;=\; |\mathcal T(T)| ,
\end{align}
and we use whichever of the two sets is more convenient, appealing to \eqref{eq:linesEqualTraces} whenever we pass from one to the other.

For example, take $K=\RR$ and
\[
T = \{ Q_1, Q_2, Q_3, Q_4 \} = \{ (0,0),\ (1,1),\ (2,2),\ (0,2) \}.
\]
The points $Q_1, Q_2, Q_3$ are collinear, lying on the line $y=x$, while $Q_4=(0,2)$ does not lie on this line. Thus $T$ determines four lines:
\begin{align*}
\ell_{123} &= \{(t,t) : t \in \RR\}, & \ell_{123}\cap T &= \{Q_1,Q_2,Q_3\}, \\
\ell_{14} &= \{(0,t) : t \in \RR\}, & \ell_{14}\cap T &= \{Q_1,Q_4\}, \\
\ell_{24} &= \{(t,2-t) : t \in \RR\}, & \ell_{24}\cap T &= \{Q_2,Q_4\}, \\
\ell_{34} &= \{(t,2) : t \in \RR\}, & \ell_{34}\cap T &= \{Q_3,Q_4\}.
\end{align*}
Hence $\mathcal L(T)=\{\ell_{123},\ell_{14},\ell_{24},\ell_{34}\}$ and
\[
\mathcal{T}(T) = \bigl\{\, \{Q_1,Q_2,Q_3\},\ \{Q_1,Q_4\},\ \{Q_2,Q_4\},\ \{Q_3,Q_4\} \,\bigr\}, \qquad |\mathcal{L}(T)| = |\mathcal{T}(T)| = 4.
\]
This example is visualized in Figure~\ref{fig:traces_example}.

\begin{figure}[htbp]
  \centering
  \includegraphics[width=0.7\textwidth]{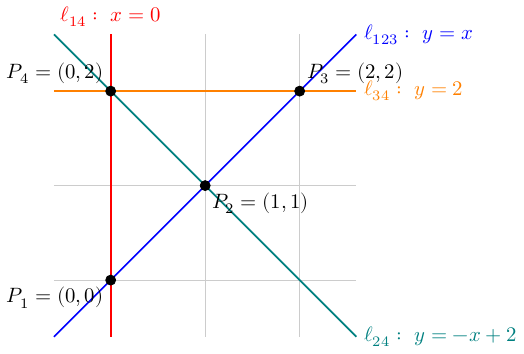}
  \caption{Example of a point set of size four with a trace set of size four.}
  \label{fig:traces_example}
\end{figure}

\section{The Toggle Game}\label{sec:ToggleGame}

The toggle game is played on a given lattice $P$. A state of the toggle game is a subset $T$ of the vertices of $P$, representing the vertices that are \emph{toggled on}. For intuition it is convenient to think of the vertices in $T$ as ``red'' and of the vertices not in $T$ as ``black''; we will use this terminology whenever $T$ is clear from the context. The toggle game starts with all vertices black, so $T=\emptyset$. The goal is to colour all vertices except $\top$ red, so $T = P \setminus \{ \top \}$. A move consists of picking a vertex $v \neq \top$ and toggling the colour of $v$ and of all its descendants. However, a vertex $v$ can only be picked if $v$ and all its descendants currently have the same colour and $\mu(v)\ne 0$. Note that the condition $\mu(v)\ne 0$ depends only on the lattice structure and not on the current state. The toggle game is won if the goal state is reached. It is not possible to ``lose'' a toggle game, since every move can be reverted. However, not all lattices allow the toggle game to be won, as we will see. A \emph{toggle sequence} is a sequence of vertices that are picked in a toggle game; it is \emph{winning} if it leads to the goal state.

Formally:

\begin{definition}\label{def:toggle-sequence}
A \emph{toggle sequence} for a lattice $P$ is a sequence of vertices defined inductively as follows. The empty sequence is a toggle sequence. If $v_{1},\ldots,v_{n}$ is a toggle sequence, then $v_{n+1} \neq \top$ may be appended if and only if $\mu(v_{n+1})\neq 0$ and $v_{n+1}$ and all its descendants have the same toggle parity, where the toggle parity after $n$ moves is
\begin{align*}
    \pi_{n}(v) & =\bigl| \{ i \in [n] : v \leq v_i \} \bigr|\bmod 2 .
\end{align*}
\end{definition}

In other words, picking $v_{n+1}$ ``toggles'' $v_{n+1}$ and all its descendants, and $v_{n+1}$ may only be picked if $v_{n+1}$ and all its descendants currently have the same parity. So picking $v_{n+1}$ either turns $v_{n+1}$ and all its descendants on ($\pi$ changes from $0$ to $1$) or turns $v_{n+1}$ and all its descendants off ($\pi$ changes from $1$ to $0$). The state after $n$ moves is $T_n=\{v: \pi_n(v)=1\}$. The toggle sequence is winning if $T_n = P \setminus \{ \top \}$.

Note that $\top$ is never picked, and that consequently $\pi_n(\top)=0$ throughout: the top element only ever serves as a reference point.

Consider for example the lattice $P$ represented in Figure~\ref{fig:toggle}. A winning toggle sequence for this lattice is $a,d,b,f,g,e,c$, as shown in Figure~\ref{fig:Example_Toggle_Game}. Note that $\bot$ must not appear in a toggle sequence, since $\mu(\bot)=0$.

\begin{figure}[htbp]
  \centering
  \begin{tabular}{l l l}
    \toprule
    \textbf{Picked} & \textbf{Descendants} & \textbf{Game State After Picking} \\
    \midrule
    $a$     & $\{ d,e,\bot \}$         & $\{a,\phantom{b,c,}d,e,\phantom{f,g,}\bot \}$ \\
    $d$     & $\{ \bot \}$             & $\{a,\phantom{b,c,d,}e,\phantom{f,g,\bot} \}$ \\
    $b$     & $\{ d,f,\bot \}$         & $\{a,b,\phantom{c,}d,e,f,\phantom{g,}\bot \}$ \\
    $f$     & $\{ \bot \}$             & $\{a,b,\phantom{c,}d,e,\phantom{f,g,\bot} \}$ \\
    $g$     & $\{ \bot \}$             & $\{a,b,\phantom{c,}d,e,\phantom{f,}g,\bot \}$ \\
    $e$     & $\{ \bot \}$             & $\{a,b,\phantom{c,}d,\phantom{e,f,}g\phantom{,\bot} \}$ \\
    $c$     & $\{ e,f,\bot \}$         & $\{a,b,c,d,e,f,g,\bot \}$ \\
    \bottomrule
  \end{tabular}
  \caption{Example of a winning toggle sequence and the corresponding game states.}
  \label{fig:Example_Toggle_Game}
\end{figure}

\subsection{Connection to the dot-algebra}\label{sec:ConnectionDot}
A winning toggle sequence translates directly into a representation in the dot-algebra. For a vertex $v \neq \top$ let $S_v$ denote the set associated with $v$, i.e.\ the set consisting of $v$ and all its descendants, and let $S_\top$ be the union of all these sets. Toggling a vertex \emph{on} corresponds to $\dotcup$ and toggling it \emph{off} corresponds to $\dotsub$; the legality condition of the toggle game is exactly the condition that the union is disjoint, respectively that the subtracted set is a subset. A toggle sequence always operates on the result obtained so far, so all parentheses are moved to the left. This corresponds to a left-linear tree, since every right child of an internal node in the syntax tree is a leaf. The goal state --- every vertex except $\top$ red --- says precisely that the set built by the sequence is $S_\top$. Note that $\top$ itself is not an element of the dot-algebra, which is why it is excluded both from the picks and from the goal state.

For the sequence $a,d,b,f,g,e,c$ of Figure~\ref{fig:Example_Toggle_Game}, in which $a,b,g,c$ are toggled on and $d,f,e$ are toggled off, this yields
\begin{align*}
     S_\top \;=\;
\left(
    \left(
        \left(
            \left(
                \left(
                    S_a\dotsub S_d
                \right)\dotcup S_b
            \right)\dotsub S_f
        \right)\dotcup S_g
    \right)\dotsub S_e
\right)\dotcup S_c .
\end{align*}
The converse is not true, however: a given representation in the dot-algebra cannot easily be transformed into a toggle sequence, or at least it is not clear how this would be done, because a representation in the dot-algebra could have arbitrary parentheses.

\subsection{The top toggle game}\label{sec:TopToggleGame}

The game can also be set up so that $\top$ may be picked as well and the goal is to colour $\top$, and only $\top$, red. We record that the two formulations are equivalent, so that nothing is lost by forbidding $\top$; after this subsection $\top$ is never picked again.

Define the \emph{top toggle game} as the toggle game with the following two changes: $\top$ may be picked like any other vertex, and the goal state is $T = \{ \top \}$.

\begin{observation}\label{obs:removeTopFromToggle}
    Let $\mathbf v = v_1,\ldots,v_n$ be a winning sequence for the top toggle game on a lattice $P$ and let $v_j$ be the last occurrence of $\top$ in $\mathbf v$. Then $v_{j+1},\ldots ,v_n$ is a winning toggle sequence for $P$.
\end{observation}
\begin{proof}
  The colour of $\top$ changes only when $\top$ itself is picked, and $\top$ is red in the goal state, so $\top$ occurs an odd number of times in $\mathbf v$; in particular $v_j$ exists. Since $\top$ can only be picked if all vertices of the lattice have the same colour, immediately before move $j$ either all vertices are black or all vertices are red. In the second case $\top$ would be black after move $j$ and, as $\top$ is not picked again, would stay black, contradicting the goal state. Hence immediately before move $j$ all vertices are black --- that is, the state is the initial one --- and immediately after move $j$ all vertices are red.

  Now compare the run of $v_{j+1},\ldots,v_n$ started from the all-red state with the run of the same sequence started from the all-black state. Complementing all colours simultaneously preserves the legality condition ``$v$ and all its descendants have the same colour'' as well as the effect of a move, so by induction on the number of moves the two runs pass through complementary states, and the second run is a legal toggle sequence. The first run ends in $\{\top\}$, i.e.\ with every vertex other than $\top$ black, so the second ends with every vertex other than $\top$ red, which is the goal state of the toggle game. Finally $\top$ does not occur among $v_{j+1},\ldots,v_n$ by the choice of $j$.
\end{proof}

\begin{lemma}\label{lem:topGameEquivalence}
  The top toggle game on a lattice $P$ can be won if and only if the toggle game on $P$ can be won.
\end{lemma}
\begin{proof}
  If $\mathbf w$ is a winning toggle sequence, then $\top,\mathbf w$ is a winning sequence for the top toggle game: $\top$ may be picked in the initial state, which is monochromatic, and afterwards the run of $\mathbf w$ from the all-red state is complementary to its run from the all-black state, as in the proof of Observation~\ref{obs:removeTopFromToggle}; it therefore ends in $\{\top\}$. The converse is Observation~\ref{obs:removeTopFromToggle}.
\end{proof}

\section{The plane toggle game in $\FF_p^2$}\label{sec:PointToggle}
Here we introduce a game similar to the toggle game, called the \emph{plane toggle game}. The motivation for this game comes from the blocking gadgets of lattices, defined below.
\subsection{Blocking Gadgets}

The plane toggle game extracts the main feature of lattices that make the toggle game unwinnable. To motivate it, we consider \emph{blocking gadgets} of lattices. These are specific sublattices that make many states unreachable by a toggle sequence from the start.

Consider again the lattice in Figure~\ref{fig:toggle}. If a toggle sequence picks $g$ first, then $g$ and $\bot$ are toggled on. After that, $\bot$ cannot be picked, since $\mu (\bot )=0$. As $\bot$ is on while all other vertices except $g$ are off, $g$ is the only vertex that can be picked. So the toggle sequence is stuck, and the only possible move reverts it to the initial state in which everything is off.

Imagine many more vertices in $\cld(\top) \cap \Par(\bot)$. As long as $\mu(\bot)=0$, all of these vertices have the same property as $g$: if a toggle sequence picks one of them first, then it is stuck and the only possible move is to re-pick that vertex. Note that for this argument it is only necessary that these elements have no other ancestors; they may have arbitrarily many other descendants. With this in mind we define:

\begin{definition}\label{def:blockingGadget}
  For an element $x$ in a lattice $P$, we say that $x$ \emph{induces a blocking gadget} if $\lev(x)=3$ and $\mu (x)=0$. In that case, the blocking gadget induced by $x$ is the interval $Q:=[x,\top]$, that is, the sublattice of $P$ consisting of exactly $x$ and all its ancestors. We say that the blocking gadget (or $x$) \emph{blocks} all vertices in $\cld(\top) \cap \Par(x)$.
\end{definition}

The following lemma formalizes the exact role of the blocking gadget, as discussed above.

\begin{lemma}\label{lem:blockingGadget}
  Let $P$ be a lattice, let $x\in P$ induce a blocking gadget $Q$, let $L$ be the set of level-$1$ elements of $P$ lying in $Q$, and let $M\subseteq L$ be the level-$1$ elements not blocked by $x$. Then in every state of every toggle sequence for $P$ it holds that whenever two elements of $L$ are red, at least one element of $M$ is red.
\end{lemma}

\begin{proof}
Let $C$ be the set of level-$2$ elements of $Q$ and put $A:=L\cup C=Q\setminus\{\top,x\}$. Note that $\Par(c)\subseteq M$ for every $c\in C$: an element of $L\setminus M$ is a parent of $x$, so no element lies strictly between it and $x$. We show the following stronger statement: Whenever two elements of $A$ are red, at least one element of $M$ is red.

Fix a toggle sequence for $P$ and consider the state at its end; since every prefix of a toggle sequence is again one, this covers every state. For $v\in A$ let $\operatorname{red}(v)$ and $\operatorname{black}(v)$ denote the number of times $v$ has been picked and thereby recoloured from black to red, respectively from red to black. Since $x$ is a descendant of every element of $A$, a picked element of $A$ always has the colour of $x$; hence
\begin{align}\label{eq:gadget-when}
  \text{$v\in A$ is picked while $x$ is black}\iff\text{that pick recolours $v$ from black to red.}
\end{align}

Since $\mu(x)=0$, the element $x$ is never picked, and $\top$ is never picked either; a pick of a vertex outside $A$ does not change the colour of $x$, because $x$ is a descendant only of the elements of $Q$. So the colour of $x$ flips exactly at the moves that pick an element of $A$. As $x$ starts black, the number of such moves made while $x$ is black exceeds the number made while $x$ is red by $1$ if $x$ is red at the end, and by $0$ if $x$ is black. By \eqref{eq:gadget-when} this says
\begin{align}\label{eq:gadget-global}
  \sum_{v\in A}\operatorname{red}(v)\;-\;\sum_{v\in A}\operatorname{black}(v)\;=\;[\,x\text{ is red}\,]\;\in\;\{0,1\}.
\end{align}

Now assume that every element of $M$ is black. An element $v\in L$ lies at level $1$, so the only elements above it are $\top$ and $v$ itself; its colour therefore changes only when $v$ is picked, and since it starts black,
\begin{align}\label{eq:gadget-level1}
  \operatorname{red}(v)-\operatorname{black}(v)=[\,v\text{ is red}\,] \qquad (v\in L);
\end{align}
in particular $\operatorname{red}(a)=\operatorname{black}(a)$ for every $a\in M$, by assumption.

The elements above $c\in C$ are $\top$, $c$ and the elements of $\Par(c)\subseteq M$, so the colour of $c$ changes exactly when $c$ or one of its parents is picked. Counting these changes with \eqref{eq:gadget-when} and using $\operatorname{red}(a)=\operatorname{black}(a)$ for $a\in\Par(c)$, the picks of the parents cancel, and only the picks of $c$ itself matter:
\begin{align}\label{eq:gadget-level2}
  [\,c\text{ is red}\,]
  &=\Bigl(\operatorname{red}(c)+\!\!\sum_{a\in\Par(c)}\!\!\operatorname{red}(a)\Bigr)
   -\Bigl(\operatorname{black}(c)+\!\!\sum_{a\in\Par(c)}\!\!\operatorname{black}(a)\Bigr)\nonumber\\
  &=\operatorname{red}(c)-\operatorname{black}(c).
\end{align}

By \eqref{eq:gadget-level1} and \eqref{eq:gadget-level2} every summand of \eqref{eq:gadget-global} satisfies $\operatorname{black}(v)\le\operatorname{red}(v)\le\operatorname{black}(v)+1$, with $\operatorname{red}(v)=\operatorname{black}(v)+1$ exactly when $v$ is red. Since by \eqref{eq:gadget-global} these differences add up to at most $1$, at most one of them is nonzero. Hence at most one element of $A$ is red; in particular at most one element of $L$ is red. Contrapositively, if two elements of $L$ are red, then some element of $M$ is red.
\end{proof}

The more vertices a blocking gadget of $x$ blocks, the more vertices we need in $\cld(\top) \setminus \Par(x)$ in order to force $\mu(x)=0$. Crucially, we can get away with $O(\sqrt{p})$ vertices in $\cld(\top) \setminus \Par(x)$ while maintaining a depth of $3$ if $|\cld(\top)|=p$, as explained below.

Suppose $|\cld(\top)| = p$ and $\cld(\top) \setminus \Par(x) = \{ a_1,a_2,\ldots,a_k \}$. Introducing, for all $1\le i<j\le k$, the $\binom{k}{2}$ vertices $c_{i,j}$ on level $2$ with parents $a_i$ and $a_j$ (which implies $\mu (c_{i,j}) = 1$) makes
\begin{align*}
    \mu(x)=-\left( 1-p+\binom{k}{2} \right) = p-1-\binom{k}{2}.
\end{align*}
Thus, in order to get $\mu(x)=0$, we need to choose $k$ such that $p-1 \leq \binom{k}{2}$; the inequality suffices because we may simply omit some of the $c_{i,j}$. Solving for $k$ yields
\begin{align}\label{eq:exactNumberMarkedPoints}
    k \geq \left\lceil \frac{1 + \sqrt{8p - 7}}{2} \right\rceil,
\end{align}
so setting $k := 1 + \left\lceil \sqrt{2p}\right\rceil$ is sufficient, since then
$\binom k2 \ge \tfrac{(1+\sqrt{2p})\sqrt{2p}}{2} = p + \tfrac{\sqrt{2p}}{2} \ge p-1$.

The goal is to construct a lattice containing so many overlapping blocking gadgets that, no matter which vertices are toggled, a toggled vertex is a blocking vertex in too many gadgets, so that after a few moves only toggling back is possible.

\subsection{The Game}
The only property of the toggle game on lattices we need to show that there is a lattice on which the toggle game cannot be won is given by the blocking gadgets. The plane toggle game extracts this property for a specific lattice with many overlapping blocking gadgets.

Intuitively, we play on the points of $\FF_p^2$, start in the state where all points are black and try to colour all points red. However, at every moment the game state must be \emph{admissible}, meaning that every trace determined by the red points must contain at least one point marked for the corresponding line. The marked points are given by a \emph{marking function} $\mathfrak{m}$, which assigns to every line $\ell$ of $\FF_p^2$ a subset $\mathfrak{m}(\ell)\subseteq \ell$ of \emph{marked} points of $\ell$. Note that a point can be marked for one line and unmarked for another line. Every line will correspond to a blocking gadget induced by an element $x$ and the marked points on that line will correspond to the elements of $\cld(\top) \setminus \Par(x)$.

Figure~\ref{fig:Point_Toggle_Example} contains two example games for the same prime number $p=3$ but different markings. In both games each line has exactly one marked point. Consider the game on the left. For this game, $T= \{ (1,0) \}$ is an admissible state, since that determines no traces. Indeed, for any plane toggle game, any point can be toggled from the initial state. From $T= \{ (1,0) \}$, a move that is not allowed would be picking $(0,0)$, since that would create the trace $\{ (0,0) , (1,0)\}$ and these two points are not marked for the corresponding line $L_{0,0}$. For the game on the right, $T=\{(0,0),(0,1),(1,1)\}$ is admissible, since it determines three lines and each of the three corresponding traces
$\mathcal{T}(T) = \{ \{(0,0),(0,1)\} , \{(0,0),(1,1)\} , \{(0,1),(1,1)\} \}$
contains a point marked for the corresponding line.

A winning plane toggle sequence for the right plane toggle game in Figure~\ref{fig:Point_Toggle_Example} is $(0,0)$, $(0,1)$, $(1,1)$, $(1,0)$, $(0,2)$, $(1,2)$, $(2,2)$, $(2,1)$, $(2,0)$, so this plane toggle game is winnable. Note that what makes the game winnable is the fact that the markings are in a way clustered, so choosing the points first that are marked on many lines gives a winning sequence. The left plane toggle game, on the other hand, is not winnable, because its marked points are more spread out. Proving this is already more difficult, because a plane toggle sequence can be arbitrarily long, toggling some points on and other points off again. However, every winning plane toggle sequence has to pass through a state in which exactly seven points are red, and by brute-forcing all these possibilities one can show that all states of size exactly seven are not admissible, proving that this specific plane toggle game is not winnable.

\begin{figure}[htbp]
  \centering
  \includegraphics[width=0.95\textwidth]{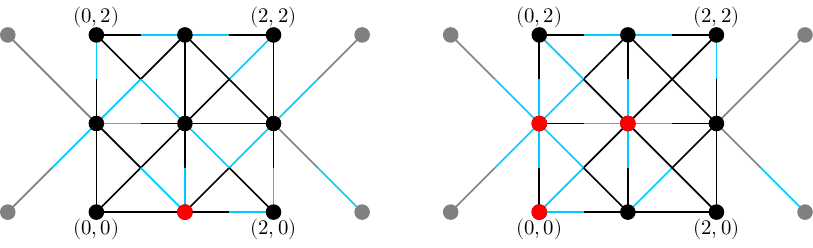}
  \caption{Two examples of admissible states in plane toggle games. Both games are over $\FF_3^2$ for two different markings. Since a point can be marked for one line and unmarked for another line, we show the marking of a point on a line next to the point. A point is marked for a line if the line is blue around that point. For example, $(1,1)$ on the right is marked for the vertical and horizontal line through it, but not for the diagonal lines through it. Both markings mark one point on each line. The gray points, which are repetitions of other points, are shown in order to have straight lines in the picture. For example, the line through $(0,1)$ and $(1,2)$ would have to go through $(2,0)$, but by identifying $(-1,0)$ with $(2,0)$ we can draw it as the straight line from $(-1,0)$ to $(2,0)$. A toggled point is shown in red.}
  \label{fig:Point_Toggle_Example}
\end{figure}

The purpose of the plane toggle game is that some plane toggle games are an easier version of the toggle game on certain lattices. We will find a plane toggle game corresponding to a toggle game on a lattice such that the plane toggle game cannot be won. This implies that the toggle game on the underlying lattice cannot be won either.

Formally, the plane toggle game $\PTG(p,\mathfrak{m})$ on the affine plane $\FF_p^2$ with marking function $\mathfrak{m}$ is defined as follows. A state of $\PTG(p,\mathfrak{m})$ is a subset $T \subseteq \FF_p^2$. The initial state is $\emptyset$ and the goal state is $\FF_p^2$; in other words, the goal is to colour all points red, starting from the state in which all points are black. Let $\mathbf q = (q_1,\dots,q_m)$ be a sequence of points of $\FF_p^2$. The \emph{game path} induced by $\mathbf q$ is the sequence of states after $i$ moves, formally
\begin{align*}
    \mathcal{G}(\mathbf q)  &= ( T_0 , T_1, \ldots , T_m ),     \\
    T_i          &= \bigl\{\, q \in \FF_p^2 \;:\; |\{ j \in \NN : 1\leq j\leq i \text{ and } q_j = q \}| \text{ is odd} \,\bigr\}.
\end{align*}
So a point is red after $i$ moves (that is, lies in $T_i$) if and only if it has been picked an odd number of times among the first $i$ moves.

A set $T \subseteq \FF_p^2$ is called \emph{admissible} for a marking $\mathfrak{m}$ if every trace of $T$ contains at least one marked point of the corresponding line, i.e.\ if for every line $\ell$ with $|T\cap \ell|\ge 2$ we have $T\cap\ell\cap\mathfrak{m}(\ell) \neq \emptyset$. A \emph{plane toggle sequence} is a sequence $q_1 , \dots , q_m \in \FF_p^2$ such that every state in the induced game path is admissible. A plane toggle sequence is \emph{winning} if $T_m = \FF_p^2$.

Note that $|T_i|$ changes by exactly $1$ in every move. In particular, every winning plane toggle sequence passes through a state with exactly $n$ red points, for every $0\le n\le p^2$.

\section{Lattices Inducing Plane Toggle Games}\label{sec:LatticeConstruction}

Here we define the lattices that induce plane toggle games.

\subsection{The Lattice $P_{p,\mathfrak{m}}$}

The lattice we use to disprove the left-linear version of the conjecture is the union of many overlapping blocking gadgets.
Let $p$ be a prime number, put
\[
  w:=\left\lceil \sqrt{2p}\right\rceil +1 ,
\]
and let $\mathfrak{m}$ be a marking function that assigns to each line $L_{i,j}$ of $\FF_{p}^{2}$ a subset $\mathfrak{m}(L_{i,j}) \subseteq L_{i,j}$ with $|\mathfrak{m}(L_{i,j})|=w$. We define a lattice $P_{p,\mathfrak{m}}$ for every such $p$ and $\mathfrak{m}$; later we show that there is a specific pair $(p,\mathfrak m)$ for which the toggle game on $P_{p,\mathfrak{m}}$ cannot be won.
The lattice $P_{p,\mathfrak{m}}$ has depth $4$ and contains the following elements.
\begin{itemize}
\item On level $0$, the global top element $\top$.
\item On level $1$ (so in $\cld(\top)$), the $p^{2}$ elements $a_{i,j}$ with $(i,j)\in\FF_p^2$, one for each point of $\FF_p^2$. We identify $a_{i,j}$ with the point $(i,j)$ whenever convenient.
\item On level $3$, the $p^{2}+p$ elements $b_{i,j}$ with $i\in\FF_p \cup \{\infty\}$ and $j\in\FF_p$, one for each line $L_{i,j}$ of $\FF_{p}^{2}$. The level-$1$ ancestors of $b_{i,j}$ are exactly the points of the line $L_{i,j}$. These elements play the role of the $\bot$-elements of the individual blocking gadgets, so there will be level-$2$ elements between $b_{i,j}$ and its level-$1$ ancestors in order to make $\mu(b_{i,j})=0$.
\item On level $2$, elements $c_{i,j,z}$ between each $b_{i,j}$ and its level-$1$ ancestors, chosen so that $\mu(b_{i,j})=0$, exactly as described for blocking gadgets above. The element $c_{i,j,z}$ has the following properties:
\begin{itemize}
    \item $\cld(c_{i,j,z}) = \{ b_{i,j} \}$;
    \item $| \Par(c_{i,j,z}) | = 2 $;
    \item $\Par(c_{i,j,z}) \subseteq \mathfrak{m}(L_{i,j}) $;
    \item $\Par(c_{i,j,z}) \neq \Par(c_{i,j,z'})$ for $z\neq z'$.
\end{itemize}
The index $z$ ranges over $\{1,\ldots,p-1\}$, so there are exactly $p-1$ such elements for every line. We assume in addition that
\begin{align}\label{eq:allMarkedUsed}
  \bigcup_{z=1}^{p-1} \Par(c_{i,j,z}) = \mathfrak{m}(L_{i,j}) ,
\end{align}
that is, every marked point of $L_{i,j}$ is a parent of at least one $c_{i,j,z}$. This is required because of the rounding in \eqref{eq:exactNumberMarkedPoints}: whenever $\binom{w-1}{2}\ge p-1$, fewer than $w$ marked points already suffice to build the $c_{i,j,z}$, and a marked point could otherwise end up being a parent of $b_{i,j}$. This would not break anything, since more marked points only make the game easier, but assuming \eqref{eq:allMarkedUsed} spares us that case. The assumption can always be met: covering all $w$ marked points requires only $\lceil w/2 \rceil$ of the pairs, and $\lceil w/2\rceil\le p-1$ for every prime $p\ge3$.
\item On level $4$, the global bottom element $\bot$ with $\Par(\bot) = \{b_{i,j}: i\in\FF_p \cup \{\infty\},\ j\in\FF_p\}$.
\end{itemize}
A point of $L_{i,j}$ is a \emph{parent} of $b_{i,j}$ exactly if it is not a parent of any $c_{i,j,z}$; by \eqref{eq:allMarkedUsed} these are exactly the unmarked points of $L_{i,j}$, and the marked points are ancestors of $b_{i,j}$ through the corresponding $c_{i,j,z}$. In either case the set of level-$1$ ancestors of $b_{i,j}$ is exactly $L_{i,j}$, as required.

The $w$ marked points of $L_{i,j}$ define $\binom{w}{2}=w(w-1)/2$ pairs, and we use $p-1$ of them, one for each $c_{i,j,z}$. This is possible, because
\[
\frac{w(w-1)}{2}=\frac{(\left\lceil \sqrt{2p}\right\rceil +1)\left\lceil \sqrt{2p}\right\rceil }{2}\geq\frac{2p+\left\lceil \sqrt{2p}\right\rceil }{2}\geq p-1 .
\]
It is not important which pairs are used, nor that the parents of $c_{i,j,z}$ can be recovered from $z$.

With this choice we indeed get, for every line $L_{i,j}$,
\[
  \mu(b_{i,j}) \;=\; -\Bigl( \mu(\top) + \sum_{a\in L_{i,j}}\mu(a) + \sum_{z=1}^{p-1}\mu(c_{i,j,z}) \Bigr) \;=\; -\bigl( 1 - p + (p-1) \bigr) \;=\; 0 ,
\]
using $\mu(a_{i,j})=-1$ for level-$1$ elements and $\mu(c_{i,j,z})=1$ for level-$2$ elements. Since moreover $\lev(b_{i,j})=3$, every $b_{i,j}$ induces a blocking gadget in the sense of Definition~\ref{def:blockingGadget}, and the blocked vertices are exactly the unmarked points of $L_{i,j}$.

That $P_{p,\mathfrak{m}}$ is indeed a lattice is shown in Appendix~\ref{sec:PpmIsLattice}.

\subsection{From Toggle Games to Plane Toggle Games}

We now relate the toggle game on $P_{p,\mathfrak{m}}$ to the plane toggle game $\PTG(p,\mathfrak{m})$.

\begin{lemma}\label{lem:toggle-to-plane}
    If the toggle game can be won on $P_{p,\mathfrak{m}}$, then the plane toggle game $\PTG(p,\mathfrak{m})$ can be won.
\end{lemma}

\begin{proof}
    Let $\mathbf{v}=v_1,\ldots,v_m$ be a winning toggle sequence for $P_{p,\mathfrak m}$, let $\pi_i$ be the toggle parity of Definition~\ref{def:toggle-sequence}, and let
    $T_i := \{\, a \in \FF_p^2 : \pi_i(a) = 1 \,\}$
    be the set of level-$1$ vertices that are red after $i$ moves, viewed as a set of points of $\FF_p^2$.

    A level-$1$ vertex has no ancestor other than $\top$, so its colour changes exactly when it is picked itself. Hence $T_i$ changes by exactly one point when $v_i$ is a level-$1$ vertex and does not change otherwise, so the restriction of $\mathbf v$ to its level-$1$ entries is a sequence of points of $\FF_p^2$ whose induced game path consists of the sets $T_i$. It starts at $T_0=\emptyset$, and since $\mathbf v$ is winning, every vertex other than $\top$ is red at the end, so it ends at $T_m=\FF_p^2$.

    It therefore suffices to show that every $T_i$ is admissible, i.e.\ that for every line $L_{k,l}$ with $|T_i \cap L_{k,l}| \ge 2$ the set $T_i \cap L_{k,l}$ contains a point marked for $L_{k,l}$. This is precisely the role of the blocking gadget induced by $b_{k,l}$, which consists of $b_{k,l}$, the elements $c_{k,l,z}$ and the points of $L_{k,l}$. Its level-$1$ elements are the points of $L_{k,l}$, and by \eqref{eq:allMarkedUsed} the ones that are not blocked are exactly the marked ones. So Lemma~\ref{lem:blockingGadget} applies and yields: whenever at least two points of $L_{k,l}$ are red, at least one red point of $L_{k,l}$ is marked for $L_{k,l}$.
\end{proof}

Note that the converse of Lemma~\ref{lem:toggle-to-plane} is not claimed: the plane toggle game might be winnable while the toggle game on $P_{p,\mathfrak{m}}$ is not. We will show that there is a marking $\mathfrak{m}$ for which the plane toggle game cannot be won, which by Lemma~\ref{lem:toggle-to-plane} implies that the toggle game on $P_{p,\mathfrak{m}}$ cannot be won either.

\section{The finite Erdős--Beck theorem}\label{sec:FiniteBeck}
Recall the statement over the reals:

\begin{theorem}[Erdős--Beck]\label{thm:beckR}
There is an absolute constant $\beta>0$ with the following property. Let $T\subseteq\RR^{2}$ be a set of $n\ge 2$ points and let $k\ge1$ be such that no line contains more than $n-k$ points of $T$. Then $|\mathcal L(T)|\ge \beta\,n\,k$.
\end{theorem}

In this form the result was conjectured by Erdős and proved by Beck \cite{beck1983lattice}, in the same issue of \emph{Combinatorica} in which Szemerédi and Trotter \cite{szemeredi1983extremal} established their incidence bound. The derivation of the former from the latter --- a dyadic decomposition of the point set according to how many of the determined lines are rich --- is by now standard, and it uses nothing about $\RR^2$ beyond the incidence bound itself and the fact that two distinct points lie on a unique line.

We need the statement over $\CC$ rather than over $\RR$, because the Nullstellensatz argument below requires an algebraically closed field of characteristic $0$. The complex Szemerédi--Trotter theorem --- $m$ points and $e$ lines in $\CC^{2}$ determine $O(m^{2/3}e^{2/3}+m+e)$ incidences, with no restriction on $m$ and $e$ --- was proved by Tóth \cite{toth2015complex}, and independently, by rather different methods, by Zahl \cite{zahl2015r4}, for whom it is a special case of an incidence theorem for two-dimensional algebraic surfaces in $\RR^{4}$. Since both ingredients of the standard derivation are therefore available in $\CC^{2}$, the Erdős--Beck theorem holds also over $\CC$ with some universal constant $\beta>0$, and it is the following form that we use below. The exact proof is given in Appendix~\ref{sec:appBeck}.

\begin{theorem}[Erdős--Beck over $\CC$]\label{thm:beckC}
There is an absolute constant $\beta>0$ with the following property. Let $T\subseteq\CC^{2}$ be a set of $n\ge 2$ points and let $k\ge1$ be such that no line contains more than $n-k$ points of $T$. Then $|\mathcal L(T)|\ge \beta\,n\,k$.
\end{theorem}

In this section we prove the following finite-field analogue of the Erdős--Beck theorem, which is a direct consequence of the theorem over $\CC$ together with Hilbert's Nullstellensatz.

\begin{theorem}\label{thm:finiteBeck}
There is a universal constant $\beta>0$ such that for every $n \in \NN$ there is a number $p_0(n) \in \NN$ with the following property. For every prime $p\geq p_0(n)$, every set $T\subseteq\FF_p^2$ with $|T|=n$, and every $k\in[n]$ such that no line contains more than $n-k$ points of $T$, the set $T$ determines at least $\beta\,n\,k$ distinct lines, i.e.\ $|\mathcal L(T)|\ge \beta n k$.
\end{theorem}

\subsection*{Step 1: from points to a linear space}

\begin{lemma}
Let $K$ be a field and $T\subseteq K^2$ finite. Then $(T,\mathcal T(T))$ is a \emph{linear space}: every block has at least $2$ elements, and every pair of distinct points of $T$ lies in exactly one common block.
\end{lemma}

\begin{proof}
Take distinct $A,B\in T$. They determine a unique line $\ell_{AB}$ in $K^2$, and $\ell_{AB}\cap T$ is a block of $\mathcal T(T)$ containing $A$ and $B$; this gives existence. If some other block also contained $A$ and $B$, it would come from a line $\ell'$ through $A$ and $B$; since $\ell_{AB}$ is the \emph{only} such line, $\ell'=\ell_{AB}$ and the two blocks coincide. This gives uniqueness.
\end{proof}

Note that both quantities in the theorem are properties of the linear space $(T,\mathcal T(T))$ alone: ``no line contains more than $n-k$ points of $T$'' means that every block has size $\le n-k$, and ``$T$ determines at least $\beta nk$ lines'' means $|\mathcal T(T)|\ge \beta nk$, by \eqref{eq:linesEqualTraces}.

\subsection*{Step 2: only finitely many linear spaces on \texorpdfstring{$[n]$}{[n]} exist}

A linear space on $[n]$ corresponds to a partition of the edge set of the complete graph $K_n$ into cliques, and the number of partitions of an $m$-element set is the $m$-th Bell number $B_m$. Hence there are at most $B_{\binom{n}{2}}$ linear spaces on $[n]$. Bounding $B_m\le m^m$ gives the crude bound $n^{n^2}$; we only need finiteness.

Fix the constant $\beta>0$ from the Erdős--Beck theorem over $\CC$. For a linear space $S$ on $[n]$ write $\mathcal T_S$ for its set of blocks. Call $S$ \emph{bad} if there is some $k\in [n]$ such that the maximum block size of $S$ is at most $n-k$ while $|\mathcal T_S| < \beta\,n\,k$; otherwise call $S$ \emph{good}.
Let $\mathfrak{B}_n = \{\, S \text{ a linear space on } [n] : S \text{ is bad} \,\}$; then $|\mathfrak{B}_n| \le n^{n^2}$, and in particular $\mathfrak B_n$ is finite.

\subsection*{Step 3: no bad linear space is realizable over \texorpdfstring{$\CC$}{C}}

Say that a linear space $S$ on $[n]$ is \emph{realizable over a field $K$} if there exist $n$ distinct points of $K^2$ inducing $S$ via the construction of Step~1. The Erdős--Beck theorem over $\CC$ says exactly that no bad linear space is realizable over $\CC$.

\subsection*{Step 4: realizability as a polynomial system over \texorpdfstring{$\ZZ$}{Z}}

Fix a linear space $([n],S)$. We encode ``$S$ is realizable over $K$'' as the solvability over $K$ of a system $F_S$ of polynomials \emph{with integer coefficients} in the variables $x_1,y_1,\dots,x_n,y_n$ together with auxiliary variables:

\begin{itemize}[leftmargin=2em]
\item \textbf{Collinear triples.} For every $\{i,j,k\}$ lying in a common block of $S$:
    \begin{align*}
        \det\begin{pmatrix} x_i & y_i & 1\\ x_j & y_j & 1 \\ x_k & y_k & 1\end{pmatrix} = 0.
    \end{align*}
\item \textbf{Non-collinear triples.} For every $\{i,j,k\}$ \emph{not} contained in a common block of $S$ we need the determinant above to be nonzero. We introduce an auxiliary variable $z_{ijk}$ and impose the Rabinowitsch-trick equation
\begin{align*}
    z_{ijk}\cdot \det\begin{pmatrix} x_i & y_i & 1\\ x_j & y_j & 1 \\ x_k & y_k & 1\end{pmatrix} - 1= 0,
\end{align*}
which is solvable for $z_{ijk}$ over any field if and only if the determinant is nonzero (choose $z_{ijk} = 1/\det$).

\item \textbf{Distinctness.} For every $i\ne j$ we introduce auxiliary variables $u_{ij},v_{ij}$ and impose $u_{ij}(x_i-x_j) + v_{ij}(y_i-y_j) - 1 = 0$. This is solvable for $u_{ij},v_{ij}$ over any field if and only if $(x_i,y_i)\ne (x_j,y_j)$.
\end{itemize}

Let $F_S=\{f_1,\dots,f_{|F_S|}\}$ be the set of left-hand sides of these polynomials and let $V$ be its set of variables. By construction, for every field $K$,
\begin{align*}
    S \text{ is realizable over } K \iff F_S \text{ has a common zero in } K.
\end{align*}
Crucially, $F_S$ does not depend on $K$.

\subsection*{Step 5: a Nullstellensatz certificate with integer coefficients}

Let $([n],S) \in \mathfrak{B}_n$. By Step~3, $F_S$ has no common zero over $\CC$. By the weak Hilbert Nullstellensatz there are $g_i\in\CC[V]$ with
\begin{align}\label{eq:Nullstellensatz}
    \sum_{i=1}^{|F_S|} g_i f_i = 1.
\end{align}
The $g_i$ may in fact be taken in $\QQ[V]$. Indeed, fix one solution $(g_i)\subseteq\CC[V]$ and let $D=\max_i \deg g_i$. The condition ``$1=\sum_i g_i f_i$ with $\deg g_i\le D$'' is equivalent to the solvability of a \emph{finite linear system} $M\mathbf c=\mathbf b$ in the unknown coefficients $\mathbf c=(c_1,\dots,c_m)$ of the $g_i$ (treated as polynomials of degree at most $D$ with undetermined coefficients): expanding $\sum_i g_i f_i$ and comparing coefficients monomial by monomial with $1$ produces one linear equation per monomial. Since each $f_i\in\ZZ[V]$ is fixed, the matrix $M$ and the vector $\mathbf b$ have entries in $\ZZ \subseteq \QQ$.
By construction, this system is consistent over $\CC$. Performing Gaussian elimination on the augmented matrix $(M\mid\mathbf b)$ uses exclusively rational arithmetic ($+,-,\times,\div$), which cannot produce an inconsistent row of the form $(0,\dots,0\mid 1)$ since a complex solution exists. Setting any free variables to $0$ and back-substituting yields a rational solution vector $\mathbf c\in\QQ^m$. Hence, there exist polynomials $g_i\in\QQ[V]$ such that $1=\sum_i g_i f_i$.

Clearing denominators (let $N_S\in\ZZ_{>0}$ be a common denominator of all coefficients occurring in the $g_i$, and put $h_i:=N_S\, g_i\in\ZZ[V]$) produces the integer identity
\begin{align*}
    \sum_{i=1}^{|F_S|} h_i f_i = N_S.
\end{align*}

\subsection*{Step 6: reduction modulo \texorpdfstring{$p$}{p}}

Let $p$ be any prime with $p\nmid N_S$. Then $N_S$ is invertible in $\FF_p$ and, reducing the identity above modulo $p$,
\begin{align*}
    \sum_{i=1}^{|F_S|} N_S^{-1} h_i f_i = 1 \qquad\text{in } \FF_p[V].
\end{align*}
Applying the weak Nullstellensatz over $\FF_p$ (or rather its trivial direction), the polynomials $N_S^{-1} h_i$ certify that $F_S$ has no common zero over $\FF_p$, hence $S$ is not realizable over $\FF_p$.

\subsection*{Step 7: conclusion}

Set $E(n) \;=\; \{\, p \text{ prime} : p \mid N_S \text{ for some } S \in {\mathfrak{B}}_n \,\}$. Since $\mathfrak{B}_n$ is finite (Step~2) and each $N_S$ is a fixed positive integer, $E(n)$ is a \emph{finite} set of primes. Put $p_0(n) := 1+\max\bigl(E(n)\cup\{1\}\bigr)$, so that $p_0(n)=2$ in the degenerate case $\mathfrak B_n=\emptyset$.

Now let $p\geq p_0(n)$ and let $T\subseteq\FF_p^2$ be any set of $n$ distinct points. By Step~6, no linear space in $\mathfrak{B}_n$ is realizable over $\FF_p$; in particular the linear space $([n],\mathcal T(T))$ induced by $T$ (Step~1) is not in $\mathfrak{B}_n$, i.e.\ it is \emph{good}. Being good means precisely: for every $k$ such that no block has size larger than $n-k$ we have $|\mathcal T(T)| \ge \beta\,n\,k$, and $|\mathcal L(T)|=|\mathcal T(T)|$ by \eqref{eq:linesEqualTraces}. This concludes the proof of Theorem~\ref{thm:finiteBeck}.

\begin{remark}
The proof gives no usable bound on $p_0(n)$. An effective $p_0(n)$ is calculated in Section~\ref{sec:HowLargeP}.
\end{remark}

\section{The Main Argument}\label{sec:MainArgument}

We show that every winning plane toggle sequence would have to visit a non-admissible state. Since this is a contradiction, no winning plane toggle sequence exists.

On a high level the argument is as follows. A winning plane toggle sequence toggles one point at a time, so at some moment exactly $n$ points are red. If these points are not almost all on one line, they determine $\Omega(n^2)$ traces by Theorem~\ref{thm:finiteBeck}, most of which are short, as many long traces would require more points. Each short trace has to contain a marked point, so there would have to be a marked point in $\Omega(n^2)$ prescribed places. Since only $O(\sqrt p)$ points are marked on each line, a union bound shows that for a randomly chosen marking such a configuration is very unlikely, and a first-moment argument produces a marking for which no such configuration exists at all.

For the remaining case, in which almost all red points lie on one line, we argue as follows. Suppose $T$ is admissible. Removing a point from $T$ may destroy admissibility, namely if the removed point was the only marked point of a trace of size at least $3$. However, if almost all points of $T$ lie on one line $L$, then most points of $T\cap L$ have the property that all their traces are either of size exactly $2$ or equal to the trace of $L$. Removing such a point yields another admissible set, since removing a point from a trace of size exactly $2$ deletes that trace altogether. Repeating this, we can thin out the red set until it has exactly $n$ points, still enough of them off the line $L$ to fall into the first case.

\subsection{Choice of the constants}

Let $\beta\in(0,1]$ be the constant from Theorem~\ref{thm:finiteBeck}; decreasing $\beta$ only weakens that theorem, so we may indeed assume $\beta\le1$. Fix an integer
\[
  c \;\ge\; \max\Bigl\{\,4,\ \bigl\lceil 8/\beta \bigr\rceil+1 \,\Bigr\}
  \qquad\text{and put}\qquad
  n:=c^{2}, \qquad s:=\Bigl\lceil \sqrt{2c/\beta} \Bigr\rceil+1 .
\]
For a prime $p$ we keep writing $w=\lceil\sqrt{2p}\rceil+1$ for the number of marked points per line; note that $w\le 3\sqrt p$ for every prime $p$.

Call a set $T\subseteq\FF_p^2$ \emph{spread} if $|T\cap \ell|\le n-c$ for every line $\ell$, i.e.\ if no line contains more than $n-n/c$ of its points.

\subsection{Spread sets have many short traces}

\begin{lemma}\label{lem:short-traces}
Let $p\ge p_0(n)$ and let $T\subseteq\FF_p^2$ be spread with $|T|=n$. Then at least $\frac{\beta}{2c}\,n^{2}$ of the traces in $\mathcal T(T)$ have size at most $s$.
\end{lemma}

\begin{proof}
Applying Theorem~\ref{thm:finiteBeck} with $k=c$ (which is legitimate, since no line contains more than $n-c$ points of $T$) and using $n=c^2$ together with \eqref{eq:linesEqualTraces} gives
\begin{align*}
    |\mathcal{T}(T)| \;=\; |\mathcal{L}(T)| \;\geq\; \beta\, n\, c \;=\; \frac{\beta}{c}\, n^2 .
\end{align*}
Assume for contradiction that at least $\frac{\beta}{2c}n^{2}$ traces had size at least $s$. Distinct traces share at most one point, so they share no pair of points, and a trace of size at least $s$ contains at least $\binom{s}{2}$ pairs. Counting pairs of points of $T$ we would get
\begin{align*}
    \binom{n}{2} \;\ge\; \binom{s}{2}\cdot\frac{\beta}{2c}n^{2} \;=\; \frac{s(s-1)}{2}\cdot\frac{\beta n^{2}}{2c}
    \;>\; \frac{2c/\beta}{2}\cdot\frac{\beta n^{2}}{2c} \;=\; \frac{n^{2}}{2} \;>\; \binom n2 ,
\end{align*}
where we used $s(s-1) > (s-1)^2 \ge 2c/\beta$. This is a contradiction, so fewer than $\frac{\beta}{2c}n^{2}$ traces have size at least $s$. Consequently at least
$\frac{\beta}{c}n^{2}-\frac{\beta}{2c}n^{2}=\frac{\beta}{2c}n^{2}$
traces have size at most $s$.
\end{proof}

\subsection{A marking without admissible spread sets}

\begin{lemma}\label{lem:marking-exists}
For every sufficiently large prime $p$ there is a marking $\mathfrak m$ with $|\mathfrak m(\ell)|=w$ for every line $\ell$ of $\FF_p^2$ such that no spread set $T\subseteq\FF_p^2$ with $|T|=n$ is admissible for $\mathfrak m$.
\end{lemma}

\begin{proof}
Choose $\mathfrak m$ at random: for each line $\ell$ independently, let $\mathfrak m(\ell)$ be a uniformly random $w$-element subset of $\ell$. Then every point of $\ell$ is marked for $\ell$ with probability $w/p$, and the markings of distinct lines are independent.

Fix a spread set $T$ with $|T|=n$ and let $p\ge p_0(n)$. By Lemma~\ref{lem:short-traces} there are at least $\frac{\beta}{2c}n^{2}$ traces of $T$ of size at most $s$. Distinct traces of size at least $2$ come from distinct lines, so the corresponding events are independent, and by the union bound the probability that a fixed trace of size at most $s$ contains a point marked for its line is at most $s\,w/p$. Hence
\[
  \Pr[\,T \text{ is admissible}\,] \;\le\; \Bigl(\frac{s\,w}{p}\Bigr)^{\frac{\beta}{2c}n^{2}} .
\]
There are at most $\binom{p^{2}}{n}\le p^{2n}$ candidates for $T$, so, using $w\le 3\sqrt p$ and $n^2/c=nc$, the expected number of admissible spread sets of size $n$ is at most
\begin{align*}
\binom{p^{2}}{n} \Bigl(\frac{s\,w}{p}\Bigr)^{\frac{\beta n^{2}}{2c}}
&\;\le\; p^{2n} \left(\frac{3s}{\sqrt{p}}\right)^{\frac{\beta n^{2}}{2c}} \\[4pt]
&\;=\; (3s)^{\frac{\beta n^{2}}{2c}} \cdot p^{\,2n-\frac{\beta n^{2}}{4c}}
\;=\; (3s)^{\frac{\beta n^{2}}{2c}} \cdot p^{\,n\left(2-\frac{\beta c}{4}\right)} .
\end{align*}
Since $c>8/\beta$, the exponent $n\bigl(2-\frac{\beta c}{4}\bigr)$ is negative, while the factor $(3s)^{\beta n^{2}/(2c)}$ does not depend on $p$. Hence the whole expression tends to $0$ as $p\to\infty$, so it is smaller than $1$ for every sufficiently large prime $p$.

Since the expected number of admissible spread sets of size $n$ is smaller than $1$, there must be a marking $\mathfrak m$ for which this number is $0$.
\end{proof}

\subsection{No winning plane toggle sequence}

\begin{proposition}\label{prop:ptg-not-winnable}
Let $p$ be a prime that is large enough for Lemma~\ref{lem:marking-exists} and satisfies in addition $p^{2}-p\ge c+1$, and let $\mathfrak m$ be a marking as provided by that lemma. Then $\PTG(p,\mathfrak m)$ cannot be won.
\end{proposition}

\begin{proof}
Suppose $\mathbf q=(q_1,\dots,q_m)$ were a winning plane toggle sequence, with game path $(R_0,\dots,R_m)$, $R_0=\emptyset$ and $R_m=\FF_p^2$. For a state $R$ let
\[
  L(R) \;:=\; \text{a line with } |R\cap L(R)| = \max_\ell |R\cap \ell| \qquad\text{(the \emph{main line} of $R$)},
\]
and let $f(R) := |R| - |R\cap L(R)|$ be the number of red points off the main line. Ties are broken arbitrarily.

\emph{Step 1: there is an admissible state with exactly $n$ points.} Every state $R_i$ is admissible, $|R_0| = 0$, $|R_m|=p^2>n$ and $|R_i|$ changes by exactly $1$ in every move, so there is an $i$ with $|R_i|=n$. Let $\tau$ be the last index with $|R_\tau|=n$. Again since $|R_i|$ changes by exactly $1$ in every move, every state $R_i$ with $i>\tau$ satisfies $|R_i|>n$.

\emph{Step 2: almost all points in $R_\tau$ are collinear.} By Lemma~\ref{lem:marking-exists} the admissible set $R_\tau$ is not spread, so some line contains more than $n-c$ of its points, i.e.\ $f(R_\tau)<c$.

\emph{Step 3: some later state has exactly $c+1$ red points off its main line.} Note that $f$ changes by at most $1$ per move: adding a point increases $|R|$ by $1$ and cannot decrease $\max_\ell|R\cap\ell|$, and removing a point decreases $|R|$ by $1$ and cannot increase $\max_\ell|R\cap\ell|$. By Step~2, $f(R_\tau)<c$, whereas $f(R_m)=p^{2}-p\ge c+1$. Hence there is an index $i>\tau$ with $f(R_i)=c+1$. Fix such an $i$, write $R:=R_i$ and $L:=L(R)$, so that
\[
  |R|\;\ge\; n , \qquad |R\setminus L| \;=\; c+1 , \qquad |R\cap L| \;=\; |R|-(c+1)\;\ge\; n-c-1 \;\ge\; 2 .
\]

\emph{Step 4: thinning out $R$.} Define $T\subseteq R$ as follows. Put into $T$:
\begin{itemize}
    \item all $c+1$ points of $R \setminus L$;
    \item all points of $R\cap L$ that lie on some line spanned by two points of $R \setminus L$;
    \item one point of $R\cap L$ that is marked for $L$;
    \item arbitrary further points of $R\cap L$ until $|T|=n$.
\end{itemize}
The second item adds at most $\binom{c+1}{2}$ points, since each of the at most $\binom{c+1}{2}$ lines spanned by $R\setminus L$ meets $L$ in at most one point. The third item is possible because $R$ is admissible and $|R\cap L|\ge2$. The construction is consistent, because
\[
  (c+1)+\binom{c+1}{2}+1 \;=\; \frac{c^{2}+3c+4}{2}\;\le\; c^{2} \;=\;n \qquad (c\ge 4),
\]
and because $|R\cap L|\ge n-(c+1)$ points of $L$ are available while only $n-(c+1)$ are needed.

\emph{Step 5: $T$ is spread.} On the main line we have $|T\cap L| = n-(c+1) \le n-c$. Any other line meets $L$ in at most one point, so it contains at most $1+(c+1)=c+2$ points of $T$, and $c+2\le c^{2}-c=n-c$ because $c\ge 4$. Hence $T$ is spread and $|T|=n$.

\emph{Step 6: $T$ is admissible.} Since $\mathbf q$ is a plane toggle sequence, $R$ is admissible. The trace $T\cap L$ has $n-(c+1)\ge2$ elements and contains a point marked for $L$ by construction. Now let $\ell\ne L$ be a line with $|T\cap\ell|\ge2$. All points of $R\setminus L$ belong to $T$, and $\ell$ meets $L$ in at most one point, so $R\cap\ell$ and $T\cap\ell$ can differ at most in that intersection point $q$, and only if $q\in R\setminus T$. Assume therefore that $\ell$ meets $L$ in a point $q\in R\setminus T$. Then $|R\cap \ell|=|T\cap\ell|+1\ge 3$, and since at most one of these points lies on $L$, the line $\ell$ contains at least two points of $R\setminus L$; but then $q$ would have been put into $T$ by the second item of Step~4, a contradiction. Hence $T\cap\ell = R\cap \ell$, and this trace contains a marked point because $R$ is admissible.

So $T$ is an admissible spread set of size $n$, contradicting Lemma~\ref{lem:marking-exists}.
\end{proof}

\begin{corollary}\label{cor:main}
There is a finite lattice on which the toggle game cannot be won. Consequently the NCI conjecture fails if only left-linear trees are allowed.
\end{corollary}

\begin{proof}
Take $p$ and $\mathfrak m$ as in Proposition~\ref{prop:ptg-not-winnable} and consider the lattice $P_{p,\mathfrak m}$ of Section~\ref{sec:LatticeConstruction}. If the toggle game could be won on $P_{p,\mathfrak m}$, then by Lemma~\ref{lem:toggle-to-plane} the plane toggle game $\PTG(p,\mathfrak m)$ could be won, contradicting Proposition~\ref{prop:ptg-not-winnable}.
\end{proof}

\section{How large is $p$?}\label{sec:HowLargeP}

In this section we prove that the size of the resulting counterexample is at most $10^{10^{10^{2215}}}$. The proof of Corollary~\ref{cor:main} is non-constructive, but only in one place. Every step except the passage from $\CC$ to $\FF_p$ in Section~\ref{sec:FiniteBeck} comes with explicit constants; the threshold $p_0(n)$ produced in Step~7 there, on the other hand, was obtained from the mere existence of a Nullstellensatz certificate and therefore carries no bound at all. Effective versions of the Nullstellensatz repair exactly this. We carry the estimate out and obtain a concrete, if astronomical, upper bound on the smallest prime $p$ --- and hence on the size of the smallest lattice --- that our argument produces.

We stress that what is needed here is an \emph{arithmetic} Nullstellensatz. Bounds on the degrees of the $g_i$ alone do not help: the quantity that has to be controlled is the integer $N_S$ obtained by clearing denominators, since $p_0(n)$ is defined through the prime divisors of the $N_S$.

\subsection{The size of the polynomial system}\label{sec:sizeSystem}

Fix $n$ and a linear space $([n],S)\in\mathfrak B_n$, and recall the system $F_S$ of Step~4 in Section~\ref{sec:FiniteBeck}. We record its parameters. Write $\nu:=|V|$ for the number of variables, $|F_S|$ for the number of polynomials, $d$ for the maximal degree and $h$ for the maximal height, i.e.\ the logarithm of the largest absolute value of a coefficient.

\begin{lemma}\label{lem:systemsize}
For $n\ge 3$ the system $F_S$ satisfies
\begin{align*}
  \nu \;\le\; n^{3}, \qquad |F_S| \;\le\; n^{3}, \qquad d \;=\; 3, \qquad h \;=\; 0 .
\end{align*}
\end{lemma}

\begin{proof}
The variables are the $2n$ coordinates $x_i,y_i$, one variable $z_{ijk}$ for each non-collinear triple and two variables $u_{ij},v_{ij}$ for each pair $i\neq j$, so
\begin{align*}
  \nu \;\le\; 2n+\binom n3+2\binom n2 \;<\; \frac{n^{3}}{6}+n^{2}+2n\;\le\;n^{3},
\end{align*}
the last inequality because $n^{2}+2n\le\frac56 n^{3}$ for $n\ge3$.
The polynomials are one per collinear triple, one per non-collinear triple and one per pair, so
\begin{align*}
  |F_S| \;\le\; 2\binom n3+\binom n2\;<\;\frac{n^{3}}{3}+\frac{n^{2}}{2}\;\le\;n^{3}.
\end{align*}
The collinearity determinants have degree $2$, the Rabinowitsch equations $z_{ijk}\det(\cdot)-1$ have degree $3$, and the distinctness equations $u_{ij}(x_i-x_j)+v_{ij}(y_i-y_j)-1$ have degree $2$; hence $d=3$. Every coefficient occurring in $F_S$ lies in $\{0,\pm1\}$, so $h=0$.
\end{proof}

\subsection{The arithmetic Nullstellensatz}\label{sec:sizeANSS}

Throughout this section $\ln$ denotes the natural logarithm and $\exp(x)=e^{x}$. We use the following sharp arithmetic Nullstellensatz. For $a\in\ZZ\setminus\{0\}$ its height is $h(a)=\ln|a|$.

\begin{theorem}[Krick, Pardo and Sombra {\cite[Thm.~1]{krick2001sharp}}]\label{thm:arithNSS}
Let $f_1,\dots,f_m\in\ZZ[x_1,\dots,x_\nu]$ have no common zero in $\CC^{\nu}$, and put $d:=\max_i\deg f_i$ and $h:=\max_i h(f_i)$. Then there are $a\in\ZZ\setminus\{0\}$ and $g_1,\dots,g_m\in\ZZ[x_1,\dots,x_\nu]$ with
\begin{align*}
  a\;=\;g_1f_1+\cdots+g_mf_m
\end{align*}
and
\begin{align*}
  \deg g_i\;&\le\;4\nu\, d^{\nu},\\
  h(a),\,h(g_i)\;&\le\;4\nu(\nu+1)\,d^{\nu}\Bigl(h+\ln m+(\nu+7)\ln\bigl((\nu+1)d\bigr)\Bigr).
\end{align*}
\end{theorem}

Applying this to $F_S$ replaces Step~5 of Section~\ref{sec:FiniteBeck}: the theorem produces the integer identity of that step directly, with $N_S:=|a|$, and in addition bounds $N_S$.

\begin{corollary}\label{cor:NSbound}
  For every $n\ge 3$ and every $S\in\mathfrak B_n$,
\begin{align}\label{eq:NSbound}
  N_S\;\le\;\exp\bigl(\exp(3n^{3})\bigr),
  \qquad\text{and consequently}\qquad
  p_0(n)\;\le\;\exp\bigl(\exp(3n^{3})\bigr).
\end{align}
\end{corollary}

\begin{proof}
By Step~3 of Section~\ref{sec:FiniteBeck} the system $F_S$ has no common zero in $\CC^{\nu}$, so Theorem~\ref{thm:arithNSS} applies. Insert the parameters of Lemma~\ref{lem:systemsize}, namely $h=0$, $d=3$, $m=|F_S|\le n^{3}$ and $\nu\le n^{3}$. Bounding the bracket crudely by
\begin{align*}
  \ln m+(\nu+7)\ln\bigl((\nu+1)d\bigr)\;\le\;3\ln n+6\nu\ln(6\nu)\;\le\;7\,n^{3}\ln\bigl(6n^{3}\bigr)
\end{align*}
and the prefactor by $4\nu(\nu+1)\le 8\nu^{2}\le 8n^{6}$, we obtain
\begin{align*}
  \ln N_S\;=\;h(a)\;\le\;8n^{6}\cdot 3^{\,n^{3}}\cdot 7n^{3}\ln\bigl(6n^{3}\bigr)
  \;=\;56\,n^{9}\ln\bigl(6n^{3}\bigr)\cdot 3^{\,n^{3}} .
\end{align*}
For $n\ge3$ one has $56\,n^{9}\ln(6n^{3})\le 3^{\,n^{3}}$, so
\begin{align*}
  \ln N_S\;\le\;3^{\,2n^{3}}\;=\;\exp\bigl(2n^{3}\ln 3\bigr)\;\le\;\exp\bigl(3n^{3}\bigr),
\end{align*}
which is the first bound in \eqref{eq:NSbound}. For the second, recall from Step~7 of Section~\ref{sec:FiniteBeck} that $p_0(n)=1+\max(E(n)\cup\{1\})$, where $E(n)$ consists of the primes dividing some $N_S$ with $S\in\mathfrak B_n$. Every such prime is at most $N_S$, so $p_0(n)\le 1+\max_{S}N_S$, and the additional $1$ is absorbed by the bound just proved.
\end{proof}

Note that the number $|\mathfrak B_n|\le n^{n^{2}}$ of bad linear spaces plays no role here: we need the \emph{largest} of the $N_S$, not their product.

\subsection{The other constraints on $p$}\label{sec:sizeOther}

Besides $p\ge p_0(n)$, the argument of Section~\ref{sec:MainArgument} imposes two further conditions on $p$, both of which turn out to be negligible. Recall the constants of Section~\ref{sec:MainArgument}: $\beta$ is the constant of Theorem~\ref{thm:finiteBeck}, and $n=c^{2}$, $s=\lceil\sqrt{2c/\beta}\rceil+1$, where $c$ was only required to satisfy $c\ge\max\{4,\lceil8/\beta\rceil+1\}$. From now on we fix the smallest admissible value,
\begin{align*}
  c\;:=\;\max\Bigl\{4,\ \bigl\lceil 8/\beta\bigr\rceil+1\Bigr\}.
\end{align*}

\begin{lemma}\label{lem:otherconstraints}
The conclusion of Lemma~\ref{lem:marking-exists} holds for every prime $p\ge p_0(n)$ with $p>(3c)^{2c}$, and this also implies the condition $p^{2}-p\ge c+1$ of Proposition~\ref{prop:ptg-not-winnable}.
\end{lemma}

\begin{proof}
The proof of Lemma~\ref{lem:marking-exists} needs the expectation bound
$(3s)^{\beta n^{2}/(2c)}\,p^{\,n(2-\beta c/4)}<1$, that is,
\begin{align}\label{eq:pcondition}
  p^{\,n\left(\frac{\beta c}{4}-2\right)}\;>\;(3s)^{\frac{\beta n^{2}}{2c}}\;=\;(3s)^{\frac{\beta nc}{2}} ,
\end{align}
where we used $n=c^{2}$. Since $c\ge 8/\beta+1$ we have $\tfrac{\beta c}{4}\ge 2+\tfrac{\beta}{4}$, so the exponent on the left of \eqref{eq:pcondition} is at least $\beta n/4$, and \eqref{eq:pcondition} holds as soon as
\begin{align*}
  p^{\,\beta n/4}\;>\;(3s)^{\beta nc/2},
  \qquad\text{that is,}\qquad
  p\;>\;(3s)^{2c} .
\end{align*}
Finally $c\ge 8/\beta$ gives $2c/\beta\le c^{2}/4$, whence $s\le c/2+2\le c$ for $c\ge4$, so $(3s)^{2c}\le(3c)^{2c}$. The condition $p^{2}-p\ge c+1$ is then immediate.
\end{proof}

Since $(3c)^{2c}$ is only exponential in $c$, whereas $p_0(n)$ is doubly exponential in $n^{3}=c^{6}$, the constraint $p\ge p_0(n)$ dominates completely.

\subsection{The resulting bound}\label{sec:sizeResult}

It remains to express $n$ in terms of the one remaining constant, the constant $A$ of the complex Szemerédi--Trotter theorem (Theorem~\ref{thm:STC}).

\begin{theorem}\label{thm:sizeEstimate}
Corollary~\ref{cor:main} holds with a prime $p$ satisfying
\begin{align*}
  p\;\le\;10^{10^{10^{2215}}}
\end{align*}
and the resulting lattice $P_{p,\mathfrak m}$ has $p^{3}+2p^{2}+2$ elements.
\end{theorem}

\begin{proof}
By the proof of Theorem~\ref{thm:beckC} in Appendix~\ref{sec:appBeck} we may take
\begin{align*}
  C=(4A)^{3},\qquad
  \beta=\min\Bigl\{\tfrac13,\ \frac{1}{16384\,C^{2}}\Bigr\}=\frac{1}{2^{26}A^{6}},
\end{align*}
the minimum being attained by the second term because $A\ge1$. Hence
\begin{align*}
  c\;=\;\bigl\lceil 8/\beta\bigr\rceil+1\;=\;\bigl\lceil 2^{29}A^{6}\bigr\rceil+1\;\le\;2^{30}A^{6},
  \qquad
  n\;=\;c^{2}\;\le\;2^{60}A^{12},
\end{align*}
and therefore $3n^{3}\le 3\cdot 2^{180}A^{36}\le 2^{182}A^{36}$. Corollary~\ref{cor:NSbound} and Lemma~\ref{lem:otherconstraints} show that any prime $p$ with
\begin{align*}
  \exp\bigl(\exp(3n^{3})\bigr)\;\ge\;p\;\ge\;\max\bigl\{p_{0}(n),\,(3c)^{2c}+1\bigr\}
\end{align*}
does the job, and by Bertrand's postulate such a prime exists below $2\max\{p_{0}(n),(3c)^{2c}+1\}$, which is again at most $\exp(\exp(3n^{3}))\le\exp(\exp(2^{182}A^{36}))$.

By \cite[Thm.~1]{toth2015complex} we may take $A=10^{60}$, so
\begin{align*}
  2^{182}A^{36}\;=\;2^{182}\cdot 10^{2160}\;<\;10^{2215},
  \qquad\text{hence}\qquad
  p\;\le\;\exp\bigl(\exp(10^{2215})\bigr).
\end{align*}
Since $\log_{10}\log_{10}\exp(\exp(Y))=Y\log_{10}e+\log_{10}\log_{10}e$ for $Y>0$, and $\log_{10}e<0.435$, this gives
\begin{align}\label{eq:doublelog}
  \log_{10}\log_{10}p\;<\;0.435\cdot 10^{2215}
  \qquad\text{and in particular}\qquad
  p\;<\;10^{10^{10^{2215}}} .
\end{align}
For the size of $P_{p,\mathfrak m}$, count its levels: one element $\top$, then $p^{2}$ points,
then $(p^{2}+p)(p-1)=p^{3}-p$ elements $c_{i,j,z}$, then $p^{2}+p$ elements $b_{i,j}$, then
$\bot$, giving $p^{3}+2p^{2}+2$ in total. As $p^{3}+2p^{2}+2<p^{4}$ and
$\log_{10}\log_{10}(p^{4})=\log_{10}4+\log_{10}\log_{10}p$, the bound \eqref{eq:doublelog} leaves
enough room to absorb the summand $\log_{10}4<0.61$, so the lattice also has fewer than
$10^{10^{10^{2215}}}$ elements.
\end{proof}

\begin{remark}
Two comments on this estimate.

First, the bound is certainly very far from best possible. Every step of the estimate was made crudely, and, more importantly, the whole route through the Nullstellensatz is wasteful: it replaces a statement about $\FF_p$ by a statement about $\CC$ and pays for the transfer with a doubly exponential factor. A direct incidence-geometric argument over $\FF_p$, if one exists in the required range, would presumably give far more.

Second, the bound is an upper bound on what \emph{this} proof produces, not a lower bound on the size of a counterexample. Determining the size of the smallest lattice on which the toggle game cannot be won remains open.
\end{remark}

\section{Open Problems}

We have shown that the NCI conjecture fails when restricted to left-linear trees. However, our proof only implies the existence of a counterexample of extremely large size. Considering weakened toggle sequences, where vertices with Möbius value $\mu = 0$ can also appear, Corollary~\ref{cor:main} shows in particular that there is a lattice for which every winning weakened toggle sequence must contain at least one vertex with Möbius value $\mu = 0$. This leads to several natural open problems:
\begin{enumerate}
    \item \textbf{Explicit constructions and size bounds:} Constructing an explicit counterexample, or establishing a lower bound on the size of any counterexample.
    \item \textbf{Quantifying and classifying $\mu = 0$ vertices:} Determining the number of $\mu = 0$ vertices required in a weakened toggle sequence (e.g., as a function of the total vertex count) and characterizing which $\mu = 0$ vertices are necessary.
    \item \textbf{Full NCI conjecture:} Determining whether this argument can be extended to refute the NCI conjecture in its full generality. This is work in progress.
\end{enumerate}

\section*{Use of AI}
Claude (Anthropic) was used to write the first four paragraphs of the introduction, parts of the preliminaries, the proof of Lemma~\ref{lem:PpmIsLattice} and Theorem~\ref{thm:beckC} and the argument that the $g_i$ in \eqref{eq:Nullstellensatz} may be taken to lie in $\QQ$. The formulation of Theorem~\ref{thm:finiteBeck} as a condition on realizable linear spaces and the formulation as an equation system is likewise due to the model. Furthermore the calculation of the explicit bound in Section~\ref{sec:HowLargeP} is done mainly with the model. Beyond this, it was used for corrections of formulation and for routine calculations. The author has verified all mathematical content and takes full responsibility for the paper.

\appendix
\section{$P_{p,\mathfrak{m}}$ is a lattice}\label{sec:PpmIsLattice}

\begin{lemma}\label{lem:PpmIsLattice}
$P_{p,\mathfrak{m}}$ is a lattice.
\end{lemma}

\begin{proof}
Taking the transitive closure of a DAG defines a partial order, so only the existence of joins and meets has to be checked, and only for incomparable pairs $r_1,r_2$. Since $\top$ and $\bot$ are comparable with everything, $r_1$ and $r_2$ lie on levels $1,2,3$. Write $\Pi_{i,j,z}:=\Par(c_{i,j,z})$ and recall that $\uparrow$ and $\downarrow$ include the element itself, so that
\begin{align*}
  \uparrow\! a_{i,j}&=\{a_{i,j},\top\}, &
  \downarrow\! c_{i,j,z}&=\{c_{i,j,z},b_{i,j},\bot\},\\
  \uparrow\! c_{i,j,z}&=\{c_{i,j,z}\}\cup\Pi_{i,j,z}\cup\{\top\}, &
  \downarrow\! b_{i,j}&=\{b_{i,j},\bot\},\\
  \uparrow\! b_{i,j}&=\{b_{i,j}\}\cup L_{i,j}\cup\{c_{i,j,z}\}_{z}\cup\{\top\}, \span\span
\end{align*}
\[
  \downarrow\! a_{i,j}=\{a_{i,j}\}\cup\{c_{k,l,z}: a_{i,j}\in\Pi_{k,l,z}\}\cup\{b_{k,l}: a_{i,j}\in L_{k,l}\}\cup\{\bot\}.
\]
Here we use that the level-$1$ ancestors of $b_{i,j}$ are exactly the points of $L_{i,j}$, that $c_{i,j,z}$ has two parents and the single child $b_{i,j}$, and that distinct $z$ give distinct parent pairs. We also use that two distinct points of $\FF_p^2$ lie on a unique line and two distinct lines meet in at most one point.

\emph{Joins.} No level-$2$ element is a common upper bound: $\downarrow\! c_{i,j,z}$ is the chain $c_{i,j,z}>b_{i,j}>\bot$, so if $r_1,r_2\le c_{i,j,z}$ then $r_1,r_2\in\{b_{i,j},\bot\}$, which are comparable. Hence every common upper bound other than $\top$ is a level-$1$ point. Now $\uparrow\! r\,\cap\,\text{level }1$ equals $\{a_{i,j}\}$, $\Pi_{i,j,z}$ or $L_{i,j}$ according as $r$ lies on level $1$, $2$ or $3$; in each case it is a single point or a subset of a single line. For two such sets belonging to \emph{incomparable} $r_1,r_2$ the intersection has at most one element: this is clear if the two sets sit on different lines, and if they sit on the same line $L_{i,j}$ the only possibilities are $\Pi_{i,j,z}\cap\Pi_{i,j,z'}$ with $z\ne z'$ (at most one point, as the pairs are distinct), the remaining same-line combinations $\Pi_{i,j,z}\subseteq L_{i,j}$ and $L_{i,j}=L_{i,j}$ occurring only for comparable or equal $r_1,r_2$. So either the intersection is empty and $r_1\vee r_2=\top$, or it is a single point $a$, which is then the least common upper bound because the only other one is $\top>a$.

\emph{Meets.} Dually, no level-$1$ point is a common lower bound, since a level-$1$ point lies below $\top$ only. Suppose some $c_{i,j,z}$ is a common lower bound. Then $r_1,r_2\in\uparrow\! c_{i,j,z}\setminus\{c_{i,j,z},\top\}=\Pi_{i,j,z}$, so $\{r_1,r_2\}=\Pi_{i,j,z}$ is the parent pair; this determines $(i,j)$ (the line through $r_1,r_2$) and $z$ uniquely, and $c_{i,j,z}$ dominates the remaining common lower bounds $b_{i,j}$ and $\bot$, so $r_1\wedge r_2=c_{i,j,z}$. Otherwise all common lower bounds lie in $\{b_{k,l}\}_{k,l}\cup\{\bot\}$, and at most one $b_{k,l}$ occurs: two distinct ones would force $r_1,r_2\in\uparrow\! b_{k,l}\cap\uparrow\! b_{k',l'}$, which by the previous paragraph contains at most one element besides $\top$. Hence the common lower bounds form a chain and $r_1\wedge r_2$ is $b_{k,l}$ if such a $b_{k,l}$ exists and $\bot$ otherwise.
\end{proof}

\section{Proof of Theorem~\ref{thm:beckC}}\label{sec:appBeck}

This appendix proves the Erdős--Beck theorem over $\CC$ from the complex Szemerédi--Trotter theorem. The argument splits into two halves, according to whether the richest line of $T$ carries a constant fraction of all points or not, and it is worth saying at the outset that the two halves are of a completely different nature. If almost all points lie on one line, the bound follows from a Cauchy--Schwarz count and uses no incidence bound at all (Lemma~\ref{lem:concentrated}). If no line is that rich, the incidence bound enters, through a dyadic decomposition of the lines according to how many points of $T$ they carry (Proposition~\ref{prop:spread}). The two halves are then combined in Section~\ref{sec:appAssembly}.

\subsection{Setup}\label{sec:appSetup}

Throughout the appendix, ``line'' means a complex affine line and $T\subseteq\CC^{2}$ is a finite set of $n$ points. For a line $\ell$ we write
\[
  m(\ell)\;:=\;|\ell\cap T| ,
  \qquad
  \mathcal L_{\ge j}\;:=\;\{\,\ell \;:\; m(\ell)\ge j\,\} ,
  \qquad
  M\;:=\;\max_{\ell} m(\ell) ,
\]
and we recall from Section~\ref{sec:preliminaries} that $\mathcal L(T)$ is the set of lines with $m(\ell)\ge2$; the traces themselves play no role in this appendix.

Only two properties of the plane $\CC^{2}$ are used besides the incidence bound: two distinct points lie on a unique line, and two distinct lines meet in at most one point. The first of these already gives an identity that we will use twice, namely that every pair of distinct points of $T$ lies on exactly one line of $\mathcal L(T)$:
\begin{align}\label{eq:pairs}
  \sum_{\ell\in\mathcal L(T)}\binom{m(\ell)}{2}\;=\;\binom{n}{2}.
\end{align}

The incidence bound itself is the following theorem. Over $\RR$ it is due to Szemerédi and Trotter \cite{szemeredi1983extremal}; the complex case, with the same exponents and with no side condition relating the number of points to the number of lines, is due to Tóth and, independently, to Zahl.

\begin{theorem}[Szemerédi--Trotter over $\CC$; Tóth \cite{toth2015complex}, Zahl \cite{zahl2015r4}]\label{thm:STC}
There is an absolute constant $A\ge1$ such that for every finite set $\mathcal P$ of points and every finite set $\mathcal E$ of lines in $\CC^{2}$,
\begin{align*}
  \bigl|\{\,(Q,\ell)\in\mathcal P\times\mathcal E \;:\; Q\in\ell \,\}\bigr|
  \;\le\; A\Bigl(|\mathcal P|^{2/3}\,|\mathcal E|^{2/3}\;+\;|\mathcal P|\;+\;|\mathcal E|\Bigr).
\end{align*}
\end{theorem}

Tóth \cite[Thm.~1]{toth2015complex} proves the bound in the form
$I\le 10^{60}m^{2/3}e^{2/3}+3m+3e$ and shows that the constant $10^{60}$ suffices;
since $3\le 10^{60}$, this gives Theorem~\ref{thm:STC} with $A=10^{60}$.

We fix this constant $A$ once and for all and put
\begin{align}\label{eq:defC}
  C\;:=\;(4A)^{3}.
\end{align}

\subsection{Rich lines are rare}\label{sec:appRich}

The incidence bound is used only through the following consequence, which bounds the number of lines carrying at least $j$ points of $T$.

\begin{lemma}[rich lines]\label{lem:rich}
For every $j\ge 2$,
\begin{align*}
  \bigl|\mathcal L_{\ge j}\bigr| \;\le\; C\Bigl(\frac{n^{2}}{j^{3}}\;+\;\frac{n}{j}\Bigr).
\end{align*}
\end{lemma}

\begin{proof}
Write $L:=|\mathcal L_{\ge j}|$ and let $I$ be the number of incidences between $T$ and $\mathcal L_{\ge j}$. Every line of $\mathcal L_{\ge j}$ carries at least $j$ points of $T$, so
\begin{align}\label{eq:Ilower}
  I\;\ge\;j\,L ,
\end{align}
while Theorem~\ref{thm:STC}, applied with $\mathcal P=T$ and $\mathcal E=\mathcal L_{\ge j}$, gives
\begin{align}\label{eq:Iupper}
  I\;\le\;A\bigl(n^{2/3}L^{2/3}+n+L\bigr).
\end{align}
We distinguish two ranges of $j$.

\emph{The range $j\ge 2A$.} Here $AL\le \tfrac12 jL$, so the term $AL$ in \eqref{eq:Iupper} can be absorbed into the left-hand side of \eqref{eq:Ilower}:
\begin{align*}
  j\,L \;\le\; A\bigl(n^{2/3}L^{2/3}+n\bigr)+\tfrac12 j\,L ,
  \qquad\text{hence}\qquad
  \tfrac12\,j\,L \;\le\; A\,n^{2/3}L^{2/3}\;+\;A\,n .
\end{align*}
At least one of the two summands on the right is at least $\tfrac14 jL$, so at least one of
\begin{align*}
  \tfrac14\,j\,L\;\le\;A\,n^{2/3}L^{2/3}
  \qquad\text{and}\qquad
  \tfrac14\,j\,L\;\le\;A\,n
\end{align*}
holds. In the first case, dividing by $L^{2/3}$ gives $L^{1/3}\le 4An^{2/3}/j$ and therefore $L\le(4A)^{3}n^{2}j^{-3}$; in the second case $L\le 4Anj^{-1}$. Since $4A\le(4A)^{3}=C$, in both cases
\begin{align*}
  L\;\le\;C\,\frac{n^{2}}{j^{3}}\;+\;C\,\frac{n}{j} .
\end{align*}

\emph{The range $2\le j<2A$.} Here we do not need the incidence bound at all. Every line of $\mathcal L(T)$ is determined by a pair of points of $T$, so $L\le\binom n2\le n^{2}/2$; on the other hand $j^{3}<8A^{3}$, whence
\begin{align*}
  L\;\le\;\frac{n^{2}}{2}\;=\;\frac{n^{2}}{j^{3}}\cdot\frac{j^{3}}{2}\;\le\;4A^{3}\,\frac{n^{2}}{j^{3}}\;\le\;C\,\frac{n^{2}}{j^{3}} .
\end{align*}
This proves the lemma in both ranges.
\end{proof}

\subsection{The concentrated case}\label{sec:appConcentrated}

We now treat the case in which one line carries most of the points. Note that the bound below is completely elementary: it uses neither Theorem~\ref{thm:STC} nor Lemma~\ref{lem:rich}, and it is valid over an arbitrary field.

\begin{lemma}\label{lem:concentrated}
Let $\ell_{0}$ be a line with $m(\ell_{0})=M$ and put $S:=T\setminus\ell_{0}$ and $s:=|S|=n-M$. If $s\ge1$, then
\begin{align*}
  |\mathcal L(T)|\;\ge\;\frac{s\,M^{2}}{n} .
\end{align*}
\end{lemma}

\begin{proof}
Let $\mathcal L_{\times}$ be the set of lines meeting both $S$ and $T\cap\ell_{0}$, and for $\ell\in\mathcal L_{\times}$ put $a_{\ell}:=|S\cap\ell|\ge1$. Every $\ell\in\mathcal L_{\times}$ contains at least one point of $S$ and at least one point of $T$ on $\ell_{0}$, so $\mathcal L_{\times}\subseteq\mathcal L(T)$, and every $\ell\in\mathcal L_{\times}$ is distinct from $\ell_{0}$ and hence meets $\ell_{0}$ in exactly one point.

\emph{First we count the pairs in $S\times(T\cap\ell_{0})$ in two ways.} Each such pair spans a unique line, which lies in $\mathcal L_{\times}$; conversely a line $\ell\in\mathcal L_{\times}$ contains exactly one point of $T\cap\ell_{0}$ and exactly $a_{\ell}$ points of $S$, so it accounts for exactly $a_{\ell}$ of these pairs. Therefore
\begin{align}\label{eq:crosscount}
  \sum_{\ell\in\mathcal L_{\times}} a_{\ell}\;=\;s\,M .
\end{align}

\emph{Next we bound the second moment of the $a_\ell$.} Every pair of distinct points of $S$ lies on exactly one line, so $\sum_{\ell\in\mathcal L_{\times}}\binom{a_{\ell}}{2}\le\binom s2$. Consequently
\begin{align}\label{eq:secondmoment}
  \sum_{\ell\in\mathcal L_{\times}}a_{\ell}^{2}
  \;=\;2\sum_{\ell\in\mathcal L_{\times}}\binom{a_{\ell}}{2}\;+\;\sum_{\ell\in\mathcal L_{\times}}a_{\ell}
  \;\le\;s(s-1)+sM
  \;\le\;s\,(s+M)
  \;=\;s\,n ,
\end{align}
where the last step uses $s+M=n$.

\emph{Finally we apply Cauchy--Schwarz.} Combining \eqref{eq:crosscount} and \eqref{eq:secondmoment},
\begin{align*}
  (s\,M)^{2}
  \;=\;\Bigl(\sum_{\ell\in\mathcal L_{\times}}a_{\ell}\Bigr)^{2}
  \;\le\;|\mathcal L_{\times}|\sum_{\ell\in\mathcal L_{\times}}a_{\ell}^{2}
  \;\le\;|\mathcal L_{\times}|\;s\,n ,
\end{align*}
and dividing by $sn$ gives $|\mathcal L(T)|\ge|\mathcal L_{\times}|\ge s M^{2}/n$, as claimed.
\end{proof}

\subsection{The spread case}\label{sec:appSpread}

It remains to handle the case in which no line is rich. Here the strategy is the one described at the beginning of the appendix: by \eqref{eq:pairs} the $\binom n2$ pairs of points of $T$ are distributed among the lines of $\mathcal L(T)$, and Lemma~\ref{lem:rich} says that the lines carrying many points are too few to absorb more than a fraction of them. The remaining pairs must therefore sit on lines carrying few points, and each such line can absorb only boundedly many pairs, so there have to be many of them.

\begin{proposition}\label{prop:spread}
Let $j$ be a power of two with $j\ge 32C$, and suppose that $n\ge 4$ and $M\le n/j$. Then
\begin{align*}
  |\mathcal L(T)|\;\ge\;\frac{n^{2}}{4j^{2}} .
\end{align*}
\end{proposition}

\begin{proof}
\emph{Step 1: the pairs carried by rich lines.} Split the lines with $m(\ell)\ge j$ into dyadic blocks according to $2^{i}\le m(\ell)<2^{i+1}$. Since $j$ is a power of two and $m(\ell)\le M$ always, only the indices $i$ with $\log_{2}j\le i\le\log_{2}M$ occur. A line in the $i$-th block satisfies $\binom{m(\ell)}{2}\le\tfrac12 m(\ell)^{2}<\tfrac12 4^{\,i+1}$ and lies in $\mathcal L_{\ge 2^{i}}$, so by Lemma~\ref{lem:rich}
\begin{align}
  \sum_{m(\ell)\ge j}\binom{m(\ell)}{2}
  &\;\le\;\tfrac12\sum_{i}4^{\,i+1}\,\bigl|\mathcal L_{\ge 2^{i}}\bigr| \nonumber\\
  &\;\le\;\tfrac12\sum_{i}4^{\,i+1}\,C\Bigl(\frac{n^{2}}{2^{3i}}+\frac{n}{2^{i}}\Bigr) \nonumber\\
  &\;=\;2C\sum_{i}\Bigl(\frac{n^{2}}{2^{i}}\;+\;n\,2^{i}\Bigr).\label{eq:twoGeometric}
\end{align}

\emph{Step 2: summing the two geometric series.} The first sum in \eqref{eq:twoGeometric} runs over $2^{i}\ge j$ and the second over $2^{i}\le M$, so
\begin{align*}
  \sum_{i}\frac{n^{2}}{2^{i}}\;\le\;\frac{2n^{2}}{j}
  \qquad\text{and}\qquad
  \sum_{i}n\,2^{i}\;\le\;2nM ,
\end{align*}
and therefore, using first $j\ge 32C$ and then $M\le n/j\le n/(32C)$,
\begin{align}
  \sum_{m(\ell)\ge j}\binom{m(\ell)}{2}
  \;\le\;\frac{4Cn^{2}}{j}\;+\;4CnM
  \;\le\;\frac{n^{2}}{8}\;+\;\frac{n^{2}}{8}
  \;=\;\frac{n^{2}}{4}.\label{eq:richPairs}
\end{align}

\emph{Step 3: many pairs are left for the poor lines.} For $n\ge4$ we have $\binom n2=\tfrac12 n(n-1)\ge\tfrac38 n^{2}$. Subtracting \eqref{eq:richPairs} from the identity \eqref{eq:pairs} therefore leaves
\begin{align*}
  \sum_{m(\ell)<j}\binom{m(\ell)}{2}
  \;=\;\binom n2\;-\;\sum_{m(\ell)\ge j}\binom{m(\ell)}{2}
  \;\ge\;\frac{3n^{2}}{8}-\frac{n^{2}}{4}
  \;=\;\frac{n^{2}}{8}
\end{align*}
pairs on the lines with $m(\ell)<j$.

\emph{Step 4: each poor line absorbs few pairs.} A line with $m(\ell)<j$ satisfies $\binom{m(\ell)}{2}\le\binom{j}{2}\le j^{2}/2$. Dividing the number of pairs on poor lines by the maximal number of pairs per poor line gives
\begin{align*}
  |\mathcal L(T)|\;\ge\;\bigl|\{\ell : m(\ell)<j\}\cap\mathcal L(T)\bigr|
  \;\ge\;\frac{n^{2}/8}{j^{2}/2}\;=\;\frac{n^{2}}{4j^{2}} ,
\end{align*}
which is the assertion.
\end{proof}

\subsection{Putting the two cases together}\label{sec:appAssembly}

\begin{proof}[Proof of Theorem~\ref{thm:beckC}]
Let $T\subseteq\CC^{2}$ with $|T|=n\ge2$ and let $k\ge1$ be such that no line contains more than $n-k$ points of $T$; that is, $M\le n-k$. Let $\ell_{0}$, $S$ and $s=n-M$ be as in Lemma~\ref{lem:concentrated}. Then
\begin{align}\label{eq:sgek}
  s\;=\;n-M\;\ge\;k\;\ge\;1 ,
\end{align}
so it suffices to bound $|\mathcal L(T)|$ from below by a constant multiple of $ns$. Observe also that $M\ge2$, since any two points of $T$ lie on a common line; together with $M\le n-k\le n-1$ this forces $n\ge3$.

Suppose first that $n=3$. Then $M=2$ and $k=1$, the three points are not collinear, and $|\mathcal L(T)|=3\ge \tfrac13\,nk$.

Suppose from now on that $n\ge4$, and fix a power of two $j$ with
\begin{align*}
  32C\;\le\; j\;\le\; 64C .
\end{align*}
We distinguish the two cases announced above.

\emph{Case 1: $M>n/j$.} By Lemma~\ref{lem:concentrated} and this lower bound on $M$,
\begin{align*}
  |\mathcal L(T)|\;\ge\;\frac{sM^{2}}{n}\;>\;\frac{s}{n}\cdot\frac{n^{2}}{j^{2}}\;=\;\frac{ns}{j^{2}}\;\ge\;\frac{ns}{4j^{2}} .
\end{align*}

\emph{Case 2: $M\le n/j$.} By Proposition~\ref{prop:spread} and $s\le n$,
\begin{align*}
  |\mathcal L(T)|\;\ge\;\frac{n^{2}}{4j^{2}}\;\ge\;\frac{ns}{4j^{2}} .
\end{align*}

So in both cases $|\mathcal L(T)|\ge ns/(4j^{2})$. Using \eqref{eq:sgek} and $j\le 64C$ this becomes
\begin{align*}
  |\mathcal L(T)|\;\ge\;\frac{n\,k}{4j^{2}}\;\ge\;\frac{n\,k}{4\,(64C)^{2}}\;=\;\frac{n\,k}{16384\,C^{2}} .
\end{align*}
Together with the case $n=3$ treated above, the theorem holds with
\begin{align*}
  \beta\;:=\;\min\Bigl\{\,\tfrac13,\;\frac{1}{16384\,C^{2}}\,\Bigr\},
\end{align*}
which by \eqref{eq:defC} is an absolute constant.
\end{proof}

\begin{remark}
Two comments on the proof. First, the same argument over $\RR$, with the real Szemerédi--Trotter theorem \cite{szemeredi1983extremal} in place of Theorem~\ref{thm:STC}, gives Theorem~\ref{thm:beckR}; the field is used nowhere else.

Second, the threshold separating the two cases has to be taken at $M\approx n/(32C)$ and not, as one might expect, at $M\approx n/2$. Indeed, if $M$ is a constant fraction of $n$, then the rich lines can carry a constant fraction of all pairs: take $n/j$ parallel lines with $j$ points each, which carry $\tfrac{n}{j}\binom j2\approx\tfrac12 nj$ pairs. The counting in Proposition~\ref{prop:spread} therefore genuinely requires $M$ to be small compared with $n$, and it is Lemma~\ref{lem:concentrated} --- which is insensitive to how large $M$ is --- that covers the remaining range.
\end{remark}

\bibliographystyle{amsplain}
\bibliography{references}

\end{document}